%% file: main.tex
\documentclass[12pt]{amsart}
\usepackage[a4paper,margin=1in]{geometry}
\usepackage{microtype}
\usepackage[cal=euler]{mathalfa}
\usepackage{tikz,tikz-cd}

\input{commands}
\input{commands_local}
\usepackage[colorlinks,allcolors=blue]{hyperref}

\def\f{\bw}
\def\hGaf{\what{\Ga^\f}}
\tikzset{close/.style={outer sep=-2pt}}
\def\chr{\boldsymbol\chi} 
\def\bbd{\bd}
\def\bbv{\bv}
\def\bbw{\bw}
\def\ds#1#2{#1^{(#2)}}
\def\bvf{\hat\bv}
\def\tef{\hat\te}
\def\KKB{K''}
\def\tp#1#2{#1^{(#2)}} 

\begin{document}
\input{title}
\input{intro}

\input{translation_quiver}
\input{homological_prop}

\input{quiver_var}
\input{motivic_classes}
\input{torus_action}
\input{nak_induction}

\input{mot_classes_QV}

\bibliography{../../../tex/papers}
\bibliographystyle{../../../tex/hamsplain}
\end{document}

%% file: commands_local.tex
\defcite\rdt{riedtmann_algebren}
\defcite\aus{auslander_representation}
\defcite\hap{happel_triangulated}
\def\BB{Bia\l{}ynicki-Birula\xspace}
\def\AR{Auslander-Reiten\xspace}

\oper{Irr}
\oper{sgn}
\oper{St}
\opr\EVar{ExpVar}
\def\f{\mathrm f} 
\def\s{\mathrm s} 
\def\bu{\mathbf u}
\def\bm{\mathbf m}
\def\loc{\fL} 
\def\aa{\bA^1}
\def\br{\mathbf r}

%% file: title.tex
\author{Sergey Mozgovoy}
\title{Translation quiver varieties}

\address{School of Mathematics, Trinity College Dublin, Ireland\newline\indent
Hamilton Mathematics Institute, Ireland}

\email{mozgovoy@maths.tcd.ie}

\begin{abstract}
We introduce a framework of translation quiver varieties which includes Nakajima quiver varieties as well as their graded and cyclic versions.
An important feature of translation quiver varieties is that the sets of their fixed points under toric actions can be again realized as translation quiver varieties.
This allows one to simplify quiver varieties in several steps.
We prove that translation quiver varieties are smooth, pure and have Tate motivic classes.
We also describe an algorithm to compute those motivic classes.
\end{abstract}

\maketitle

%% file: intro.tex
\section{Introduction}
Nakajima quiver varieties play a prominent role in the geometric interpretation of irreducible integrable representations of Kac-Moody Lie algebras associated to a quiver \cite{nakajima_instantons}.
Considering fixed point sets of quiver varieties under a torus action one obtains graded quiver varieties
\cite{nakajima_quiverb,nakajima_quivere},
which
have important applications in the study of quantum affine algebras and cluster algebras \cite{nakajima_quivere,nakajima_quiverf,qin_quantum,kimura_graded}.

In this paper we propose a framework of translation quiver varieties which includes Nakajima quiver varieties as well as their graded and cyclic versions.
It also includes generalized Nakajima quiver varieties introduced in \cite{scherotzke_generalized}.
The key property of translation quiver varieties is that the sets of their fixed points under toric actions can be again realized as translation quiver varieties.
This implies that we can perform localization and apply \BB decomposition without ever leaving the realm of translation quiver varieties.

A translation quiver is a triple $(\Ga,\ta,\si)$, where $\Ga$ is a quiver and 
$$\ta:\Ga_0\to\Ga_0,\qquad \si:\Ga_1\to\Ga_1$$
are bijections such that,
for any arrow $a:i\to j$ in $\Ga$, the arrow $\si a$ goes from $\ta j$ to $i$. 
One calls ~\ta a translation and \si a semitranslation (note that the pair $(\ta,\si^2)$ defines an automorphism of the quiver $\Ga$).
The above notion is closely related to translation quivers in \AR theory (see ~\S\ref{AR theory}), where \Ga is the \AR quiver of a category and $\ta$ is its \AR translation.
Most of our translation quivers can be constructed in the following way.
Let $(Q,\ta)$ be a quiver with an automorphism.
We define its \ta-twisted double quiver $\Ga=Q^\ta$ by adding to the quiver~$Q$ new arrows $a^*:\ta j\to i$, for every arrow $a:i\to j$ in $Q$.
Then we define $\si(a)=a^*$ and $\si(a^*)=\ta(a)$.

For example, if $\ta=\id$, then we obtain the usual double quiver $\bar Q$.
One defines Nakajima quiver varieties as moduli spaces of semistable representations of $\bar Q$ subject to certain relations.
As we will see later (see Example \ref{graded qv as tqv}), one can realize graded and cyclic quiver varieties as moduli spaces of semistable representations of some \ta-twisted double quiver $\Ga=Q^\ta$ subject to the relation
$$\fr=\sum_{a\in Q_1}(a\cdot \si a-\si a\cdot\ta a).$$
Generally, given a translation quiver $\Ga=Q^\ta$, we consider its mesh algebra $\Pi=\Pi(\Ga)=k\Ga/(\fr)$ and,
for any stability parameter $\te\in\bR^{\Ga_0}$ and finite dimension vector $\bv\in\bN^{\Ga_0}$,
define the translation quiver variety $\cM_\te(\Pi,\bv)$ to be the moduli space of \te-semistable $\Pi$-modules having dimension vector~$\bv$.

We also define (framed) translation quiver varieties $\cM(\bv,\bw)$, for $\bv,\bw\in\bN^{\Ga_0}$, as a particular case of the above construction.
The vector $\bw$ is used to construct a new (framed) translation quiver $\hGaf\sps\Ga$. 
Then we define $\cM(\bv,\bw)=\cM_{\tef}(\Pi(\hGaf),\bvf)$ for a certain extension $\bvf$ of the vector $\bv$ and a certain stability parameter $\tef$ on $\hGaf$.
We define the nilpotent translation quiver variety $\cL(\bv,\bw)\sbs\cM(\bv,\bw)$ to be the subvariety parametrizing nilpotent representations.

After developing some homological algebra of translation quivers we will prove that $\cM(\bv,\bw)$ is smooth and admits \BB decompositions with respect to some toric actions.
Then we prove

\begin{theorem}
The translation quiver variety $\cM(\bv,\bw)$ is smooth and has dimension $$\bw\cdot(\bv+\bv^\ta)-\hi(\bv,\bv+\bv^\ta),$$
where $\hi$ is the Euler-Ringel form of $Q$.
Both varieties $\cM(\bv,\bw)$ and $\cL(\bv,\bw)$ are pure and their motivic classes are related by
$$[\cL(\bv,\bw)]\dual=\bL^{-\dim\cM(\bv,\bw)}[\cM(\bv,\bw)],$$
where $[\cL(\bv,\bw)]\dual$ denotes the dual motivic class (see \S\ref{motivic theory}) and $\bL=[\bA^1]$.
\end{theorem}

Nakajima quiver varieties as well as their graded and cyclic versions are polynomial-count by the results of \cite{nakajima_quivere}.
The same approach can be used to show that they have Tate motivic classes, meaning polynomials in $\bL^{\pm1}$.
In the case of Nakajima quiver varieties
one can actually obtain explicit formulas for their counting polynomials \cite{hausel_kac,mozgovoy_fermionic}
as well as for their motivic classes \cite{wyss_motivic}.
In this paper we prove that translation quiver varieties
also have Tate motivic classes
and give at the same time a new way to compute motivic classes of graded and cyclic quiver varieties.
Ideally one would like to prove that translation quiver varieties have cellular decompositions.

\begin{theorem}
The translation quiver varieties $\cM(\bv,\bw)$ and $\cL(\bv,\bw)$ have Tate motivic classes.
\end{theorem}

We prove the above statement for general translation quiver varieties $\cM_\te(\Pi,\bv)$, assuming they are smooth.
The proof of the theorem is somewhat convoluted, hence I would like to briefly comment on it.
First, we apply \BB decomposition to reduce the question to translation quiver varieties over the so-called repetition quiver (see Theorem \ref{tate class}).
Then we apply a wall-crossing formula (see Theorem \ref{wall-cross2}) to reduce the question to the motivic class of the stack of all representations (not necessarily semistable) of the mesh algebra.
Then we show that this motivic class can be related to the motivic class of the stack of representations of the Jacobian algebra (given by a quiver and a potential) associated with a translation quiver~\S\ref{jac of tq}.
This motivic class is in turn related to the motivic class of the stack of pairs $(M,\vi)$, where $M$ is a quiver representation and $\vi:M\to M^\ta$ is a homomorphism (see Proposition \ref{Pi to J}).
Finally, we compute this motivic class for a class of translation quivers, including repetition quivers (see Theorem \ref{calculation}).

In view of the last theorem it is natural to ask if translation quiver varieties $\cM(\bv,\bw)$ always have a cellular decomposition.
The localization techniques of this paper allow one to significantly simplify the quiver under consideration.
For example, one can show that this quiver can be reduced to a union of trees.
It is interesting to note that in a similar situation of quiver varieties for quivers without relations cellular decompositions always exist 
\cite{reineke_cohomology,engel_smooth}.


\subsection*{Acknowledgments}
I would like to thank Victor Ginzburg and Hiraku Nakajima for helpful comments on a draft version of the paper. 
I am grateful to the anonymous referee for many useful suggestions.

%% file: translation_quiver.tex
\section{Translation quivers}

\subsection{Conventions}
A quiver $Q$ is a tuple $(Q_0,Q_1,s,t)$, where $Q_0$ is the set of vertices, $Q_1$ is the set of arrows, and $s,t:Q_1\to Q_0$ are source and target maps, respectively.
We define the opposite quiver $Q^\op$ to be the quiver with $Q^\op_0=Q_0$, $Q^\op_1=Q_1$ and the roles of $s$ and $t$ interchanged.
For any arrow $a\in Q_1$ with $s(a)=i$ and $t(a)=j$, we write $a:i\to j$ and say that ~$a$ goes from $i$ to $j$.
For any $i,j\in Q_0$, let $Q(i,j)$ denote the set of arrows from $i$ to $j$.
We define a path $u$ in $Q$ to be a sequence of arrows $$i_0\xto{a_1}i_1\to\dots\to i_{n-1}\xto{a_n}i_n,\qquad n\ge0,$$ 
in which case we write $u=a_n\dots a_1$.
We call it a cycle if $i_0=i_n$.
We denote the trivial path at $i\in Q_0$ by ~$e_i$.
We define the path algebra $kQ$ over a field $k$ to be the associative algebra having the basis consisting of paths in $Q$ and having the product induced by concatenation of paths.

We define a morphism of quivers $f:\tl Q\to Q$ to be a
pair of maps 
$$(f_0:\tl Q_0\to Q_0,f_1:\tl Q_1\to Q_1)$$
such that, for any arrow $a:i\to j$ in $\tl Q$, we have an arrow $f_1a:f_0i\to f_0j$ in $Q$.
We will usually denote both $f_0$ and $f_1$ by $f$.
We will say that $f$ is an isomorphism if it is a bijection on vertices and arrows.
An isomorphism $f:Q\to Q$ will be called an automorphism of the quiver $Q$.
We denote the group of all automorphisms of $Q$ by $\Aut(Q)$.

Given a quiver $Q$ and a set $\La$, we consider a quiver $Q\xx\La$ that consists of copies of $Q$ for every $n\in\La$.
We denote its vertices and arrows, respectively, by
$$(i,n)=i_n=i[n],\qquad (a,n)=a_n=a[n]$$
depending on the context, for $i\in Q_0$, $a\in Q_1$, $n\in\La$.
For any arrow $a:i\to j$ in $Q$ and $n\in\La$, the arrow $a_n$ goes from $i_n$ to $j_n$.

\subsection{Translation quivers}
We define a (stable) \idef{translation quiver} to be a triple $(\Ga,\ta,\si)$, where ~$\Ga$ is a quiver and
$$\ta:\Ga_0\to \Ga_0,\qquad \si:\Ga_1\to \Ga_1$$
are bijections
such that for any arrow $a:i\to j$, we have $\si a:\ta j\to i$.
The map \ta is called a \idef{translation} and the map \si is called a \idef{semitranslation}.
There is a quiver automorphism~ $(\ta,\si^2)$ of $\Ga$ which we will sometimes denote as $\ta$ by abuse of notation.
A quiver morphism $f:\tl\Ga\to\Ga$ between two translation quivers is called a translation quiver morphism if it commutes with $\ta$ and $\si$.
We define the opposite translation quiver to be $(\Ga^\op,\ta\inv,\si\inv)$.

We define a \idef{partial translation quiver} to be a triple as above, where $\ta$ and $\si$ are only partially defined, meaning that we have a subset $\Ga'_0\sbs\Ga_0$ and injective maps
$$\ta:\Ga'_0\to\Ga_0,\qquad \si:\Ga''_1\to\Ga_1,\qquad \Ga''_1=\sets{a:i\to j}{j\in\Ga'_0},$$
such that \si induces a bijection $\Ga(i,j)\to\Ga(\ta j,i)$
for all $i\in\Ga_0$ and $j\in \Ga'_0$.

\subsection{Relation to Auslander-Reiten theory}
\label{AR theory}
Translation quivers appeared originally in Auslander-Reiten theory \rdt (see also \cite{ringel_tame,happel_triangulated,auslander_representation}).
Our translation quivers correspond to stable translation quivers and our 
partial translation quivers correspond to translation quivers in \AR theory.
More precisely, one considers a  quiver $\Ga$, a subset $\Ga'_0\sbs\Ga_0$ and an injective map $\ta:\Ga'_0\to\Ga_0$
such that, for any $i\in\Ga_0$ and $j\in\Ga'_0$, the number of arrows $i\to j$ is equal to the number of arrows $\ta j\to i$.
This means that we can construct an injective map $\si:\Ga''_1\to\Ga_1$
satisfying the above property.
In \rdt one also requires $\Ga$ to have no loops and no multiple arrows.
This implies that $\si$ is uniquely determined by the map \ta.

\begin{example}[McKay quiver]
The following example of a translation quiver structure on a McKay quiver appears in \cite{auslander_rational} (see also \cite[5.6.4]{ginzburg_lecturesa}).
Let $G\sbs\GL_2(\bC)$ be a finite subgroup and $V=\bC^2$ be the corresponding $G$-representation.
Define the McKay quiver $\Ga$ of the representation ~$V$ to have vertices corresponding to isomorphism classes of irreducible $G$-representations.
For any $L,M\in\Ga_0$, define the number of arrows from $L$ to $M$ to be the dimension of $\Hom_G(L,V\ts M)$.
Define the translation map
$$\ta:\Ga_0\to\Ga_0,\qquad L\mto \La^2 V\ts L,$$
where $\La^2V$ is the wedge product of $V$ and is a $1$-dimensional $G$-representation.
Note that the projection $V\ts V\to\La^2V$ induces an isomorphism $V\iso V\dual\ts\La^2V$.
Therefore, for any $L,M\in\Ga_0$, we have
$$\Hom_G(\ta M,V\ts L)
\iso\Hom_G(V\dual\ts\La^2V\ts M,L)
\iso\Hom_G(V\ts M,L).$$
This implies that $\dim\Hom_G(\ta M,V\ts L)=\dim\Hom_G(L,V\ts M)$, hence we have a translation quiver.
It is proved in \cite{auslander_rational} that the McKay quiver $\Ga$ with the above translation is isomorphic to the AR quiver (see below) of the category of reflexive modules over $\bC\pser{x,y}^G$ if $G\sbs\GL_2(\bC)$ has no pseudo-reflections.
\end{example}

\begin{remark}[AR quivers]
Given a finite-dimensional algebra $A$ over an algebraically closed field $k$, let $\mmod A$ be the category of finite-dimensional left $A$-modules.
One constructs its AR quiver $\Ga_A$ with vertices that are isomorphism classes of indecomposable objects in $\mmod A$.
The number of arrows from $[X]$ to $[Y]$ is defined to be the dimension of the space of irreducible morphisms
$\Irr(X,Y)=\rad(X,Y)/\rad^2(X,Y)$.
One defines $\ta$ to be the AR translation \aus, defined for all non-projective indecomposable modules
and having as an image the set of all non-injective indecomposable modules.
One has $\dim\Irr(X,Y)=\dim\Irr(\ta Y,X)$,
implying that ~$\Ga_A$ is a partial translation quiver.
If $A$ is of finite representation type, then there are no multiple arrows between vertices of $\Ga_A$ \rdt (see also \aus).
In order to obtain a (stable) translation quiver one needs to consider instead the AR quiver $\hat\Ga_A$ of the derived category $D^b(A)=D^b(\mmod A)$ and assume that $A$ has a finite global dimension~ \cite{happel_triangulated}.
The translation functor $\ta:D^b(A)\to D^b(A)$ is defined to be $\ta=\nu[-1]$, where $\nu=DA\ts^L_A-$ is the left derived functor of the Nakayama functor $X\mto D\Hom_A(X,A)$ and $DV=\Hom_k(V,k)$.
Note that $\nu$ is the Serre functor of the category $D^b(A)$, meaning that $\Hom(X,Y)\iso D\Hom(Y,\nu X)$ for all $X,Y\in D^b(A)$.
\end{remark}

\begin{remark}[Stabilization of partial translation quivers]
\label{stab1}
If $A$ is a hereditary algebra, then indecomposable objects in $D^b(A)$ are shifts of indecomposable objects in $\mmod A$.
In this case there is a construction that allows one to reconstruct the AR quiver $\hat\Ga_A$ of $D^b(A)$ from the AR quiver $\Ga_A$ of $\mmod A$ and the correspondence between projective and injective modules (given by the Nakayama functor).
Let $\Ga$ be a partial translation quiver with $\ta$ defined on $\Ga'_0\sbs\Ga_0$.
Let us call the vertices in $\Ga^{\bp}_0=\Ga_0\ms\Ga'_0$ projective and the vertices in $\Ga^\bi_0=\Ga_0\ms\ta(\Ga'_0)$ injective.
Let $\Ga^\bp$ and $\Ga^\bi$ be the corresponding full subquivers of $\Ga$ and assume that
we have a quiver isomorphism $\nu:\Ga^\bp\to \Ga^\bi$.
Assume also that there are no arrows $x\to p$ for $p\in\Ga^\bp_0$, $x\notin\Ga^\bp_0$ and no arrows $i\to x$ for $i\in\Ga^\bi_0$, $x\notin\Ga^\bi_0$.
We construct a translation quiver $\hat\Ga$ by adding arrows to the quiver $\Ga\xx\bZ$ consisting of copies of~\Ga for every $n\in\bZ$.
Let us denote by $x[n]$ the vertex $(x,n)$ and 
define 
$$\ta (x[n])=
\begin{cases}
(\ta x)[n]& x\in\Ga'_0\\
(\nu x)[n-1]& x\in\Ga^\bp_0.
\end{cases}
$$
For any two projective vertices $p,q$ we add arrows
$$\hat\Ga(\ta p[n],q[n])=\Ga(q,p).$$
Then we have bijections
$$\si:\hat\Ga(q[n],p[n])=\Ga(q,p)\isoto
\hat\Ga(\ta p[n],q[n]),$$
$$\si:\hat\Ga(\ta p[n],q[n])
=\Ga(q,p)\xto\nu\Ga(\nu q,\nu p)
=\hat\Ga(\ta q[n],\ta p[n])$$
and obtain a semitranslation $\si$ on $\hat\Ga$. This makes $\hat\Ga$ a translation quiver.

In particular, let $A$ be the path algebra of a quiver,
let $\Ga=\Ga_A$ be its AR quiver
and $\nu:\Ga^\bp\to\Ga^\bi$ be the isomorphism induced by the Nakayama functor.
Then $\hat\Ga$ is isomorphic to the AR quiver of $D^b(A)$ (\cf \cite{happel_triangulated}).
\end{remark}

\begin{example}
Let $Q$ be a quiver of the form $1\to 2$ and let $A=kQ$ be its path algebra.
There are three indecomposable modules $S_1, S_2, P$
having dimension vectors $(1,0)$, $(0,1)$, $(1,1)$, respectively. The modules $P,S_2$ are projective and the modules $P,S_1$ are injective.
The \AR quiver $\Ga_A$ has the form $S_2\xto bP\xto aS_1$ with $\ta S_1=S_2$ and $\si(a)=b$.
The Nakayama functor satisfies $\nu P=S_1$, $\nu S_2=P$, $\nu b=a$.
The quiver $\hat\Ga_A$ has arrows from $\Ga_A\xx\bZ$ as well as arrows in $\hat\Ga_A(\ta P[n],S_2[n])\iso \Ga_A(S_2,P)=\set b$, meaning that there are arrows $S_1[n-1]\to S_2[n]$.

\begin{ctikzcd}
\ar[rr,-,dotted]&&P\ar[dr]&&S_2[1]\ar[dr]&&S_1[1]&
\ar[l,dotted,-]\\
\ar[r,-,dotted]&S_2\ar[ur]&&S_1\ar[ur]&&P[1]\ar[ur]
&&\ar[ll,dotted,-]
\end{ctikzcd}

\end{example}

\subsection{Translation quivers with a cut}
\label{cuts}
We define a \idef{cut} of a translation quiver $\Ga$ to be a subset $\Ga^+_1\sbs\Ga_1$ such that $\Ga_1=\Ga^+_1\sqcup \si \Ga^+_1$.
A translation quiver \Ga has a cut if and only if every \si-orbit in $\Ga_1$ is either infinite or has an even number of elements.
The quiver $\Ga^+$ with the
set of vertices $\Ga^+_0=\Ga_0$ and the set of arrows $\Ga^+_1$ will be also called a cut of $\Ga$.
Note that $\si^2(\Ga^+_1)=\Ga^+_1$ and the quiver automorphism $\ta:\Ga\to\Ga$ (with $\ta=\si^2$ on arrows) restricts to an automorphism $\ta:\Ga^+\to \Ga^+$.
Conversely, we can associate a translation quiver with a cut to any quiver with an automorphism.

\begin{definition}
Given a pair $(Q,\ta)$, where $Q$ is a quiver and $\ta:Q\to Q$ is a quiver automorphism,
we define the \idef{twisted double quiver} $Q^\ta$ having the set of vertices $Q^\ta_0=Q_0$ and the set of arrows
$Q^\ta_1=Q_1\sqcup Q_1^*$,
where $Q_1^*$ consists of new arrows $a^*:\ta j\to i$ for arrows $a:i\to j$ in~$Q$.
We equip $Q^\ta$ with the structure of a translation quiver
such that $\ta:Q^\ta_0\to Q^\ta_0$ is defined to be $\ta:Q_0\to Q_0$ 
and $\si:Q^\ta_1\to Q^\ta_1$ is defined to be
$$\si(a)=a^*:\ta j\to i,\qquad \si(a^*)=\ta(a):\ta i\to \ta j,$$
for all $a:i\to j$ in $Q$.
This translation quiver has a cut $Q_1\sbs Q^\ta_1$.
\end{definition}


The above constructions give a $1-1$ correspondence between translation quivers with a cut and quivers with an automorphism.
In most situations our translation quivers will be equipped with a cut.

\begin{remark}[Partial cuts]
\label{partial cut}
Let $(\Ga,\ta,\si)$ be a partial translation quiver with $\ta$ defined on $\Ga'_0\sbs\Ga_0$ and assume that $\ta\Ga'_0=\Ga'_0$.
Then \ta and \si induce a translation quiver structure on the full subquiver $\Ga'\sbs\Ga$ with the set of vertices $\Ga'_0$.
Given a cut $\Ga'^+_1$ of $\Ga'$,
let $\Ga^+_1\sbs\Ga_1$ be the set consisting of arrows in $\Ga'^+_1$ and arrows $a:i\to j$ such that $i\notin\Ga'_0$ and $j\in\Ga'_0$.
Then the set of all arrows incident with vertices in $\Ga'_0$ is equal to $\Ga^+_1\sqcup\si\Ga^+_1$.
We call $\Ga^+_1$ a (partial) cut of $\Ga$.
\end{remark}


\begin{example}[Double quiver]
\label{double}
Let $Q$ be a quiver and $\ta=\Id$ be the identity automorphism of~$Q$.
Then $\bar Q=Q^\ta$ is called the \idef{double quiver} of $Q$.
It has the same vertices as $Q$ and arrows $a:i\to j$,
$a^*:j\to i$ for all arrows $a:i\to j$ in $Q$.
The semitranslation $\si$ of $\bar Q$ is given by
$\si(a)=a^*$ and $\si(a^*)=a$ for all arrows $a$ in $Q$.
Conversely, if $\Ga$ is a translation quiver such that $\ta=\id$ and the number of loops at every vertex is even, then $\Ga\iso\bar Q$ for some quiver $Q$.
\end{example}

\begin{example}[Repetition quiver]
\label{repetition}
Given a quiver $Q$, we construct the \idef{repetition quiver} $\Ga=\bZ Q$ with the set of vertices $Q_0\xx\bZ$ and with arrows
$$a_n:(i,n)\to (j,n),\qquad a_n^*:(j,n-1)\to(i,n)$$
for all arrows $a:i\to j$ in $Q$ and $n\in\bZ$.
We consider $Q$ as a full subquiver of $\bZ Q$ by identifying $i\in Q_0$ with $(i,0)$.
Define 
$$\ta(i,n)=(i,n-1),\qquad \si(a_n)=a_n^*,\qquad
\si(a_n^*)=a_{n-1}.$$
There is a cut $\Ga^+_1\sbs\Ga_1$ consisting of arrows $a_n$, for $a\in Q_1$ and $n\in\bZ$.
Note that $\ta(a_n)=\si^2(a_n)=a_{n-1}$.
Alternatively, we can first consider the quiver $\Ga^+=Q\xx\bZ$
with an automorphism 
$$\ta:\Ga^+\to\Ga^+,\qquad
\ta(i,n)=(i,n-1),\qquad
\ta(a_n)=a_{n-1},$$
and then define $\bZ Q=(\Ga^+)^\ta$.
If $Q$ is a Dynkin quiver, then the \AR quiver of $D^b(kQ)$ is isomorphic to the repetition quiver $\bZ Q$ \hap.
\end{example}

There is a simple characterization of translation quivers isomorphic to repetition quivers.

\begin{proposition}
\label{repet charact}
Let $\Ga$ be a translation quiver with a cut $\Ga^+$.
Assume that there exists a subquiver $Q\sbs\Ga^+$ such that $\Ga^+=\bigsqcup_{n\in\bZ}\ta^n Q$.
Then we have an isomorphism of translation quivers $\bZ Q\iso\Ga$.
\end{proposition}
\begin{proof}
By our assumption, there is an isomorphism of quivers $\vi:Q\xx\bZ\to \Ga^+$, where
$$i_n=\ta^{-n}i_0\mto \ta^{-n}i,\qquad a_n=\ta^{-n}a_0\mto \ta^{-n}a,\qquad
i\in Q_0,\, a\in Q_1,\, n\in\bZ.$$
By construction it preserves the translation.
Therefore it induces an isomorphism of translation quivers $(Q\xx\bZ)^\ta\to(\Ga^+)^\ta$.
But we have seen that $(Q\xx\bZ)^\ta\iso\bZ Q$ and 
$(\Ga^+)^\ta\iso\Ga$.
\end{proof}

\sec[Localization quivers]
\label{localiz}
Let $Q$ be a quiver, $\La$ be an abelian group and $\bd:Q_1\to\La,\, a\mto \bbd_a$, be a map.
We define a new quiver $\tl Q=\loc_\bd(Q)$, called the \idef{localization quiver}, with the set of vertices $Q_0\xx\La$ and with arrows
$$a_n:(i,n-\bbd_a)\to (j,n)\qquad
\forall (a:i\to j)\in Q_1,\, n\in\La.$$


Let $\ta:Q\to Q$ be an automorphism and assume that $\bd:Q_1\to\La$ is $\ta$-invariant.
For any $\om\in\La$, we define an automorphism $\ta_\om:\tl Q\to\tl Q$
$$\ta_\om(i,n)=(\ta i,n-\om),\qquad \ta_\om(a_n)=(\ta a)_{n-\om}.$$

\medskip
Let $(\Ga,\ta,\si)$ be a translation quiver and $\bd:\Ga_1\to\La$ be a map such that $\om=\bbd_a+\bbd_{\si a}\in\La$ is independent of $a\in \Ga_1$.
We will call such $\bd$ a \idef{weight map} having the \idef{total weight} $\om$.
Recall that we have a quiver automorphism $(\ta,\si^2)$ of $\Ga$ which we denote as $\ta$ by abuse of notation (in particular, $\ta$ acts as $\si^2$ on arrows).
Note that $\bbd_a=\bbd_{\ta a}$ for all $a\in \Ga_1$, hence $\bd$ is automatically $\ta$-invariant.
We equip the localization quiver $\tl\Ga=\loc_\bd(\Ga)$
with the translation quiver structure
$$\ta(i,n)=(\ta i,n-\om),\qquad
\si(a_n)=(\si a)_{n-\bbd_a}:(\ta j,n-\om)\to(i,n-\bbd_a).$$
If $\Ga$ has a cut $\Ga^+$, then the weight map is uniquely determined by the $\ta$-invariant map $\bd^+:\Ga^+_1\to\La$ and the total weight $\om\in\La$.
We will sometimes denote $\loc_\bd(\Ga)$ as $\loc_{\bd^+}^\om(\Ga)$.
We define the cut $\tl\Ga^+$ of $\tl\Ga$ that consists of arrows $a_n$ for $a\in\Ga^+_1$ and $n\in\La$.
Note that $\tl\Ga^+=\loc_{\bd^+}(\Ga^+)$ and its translation automorphism is given by $\ta_\om$ defined earlier.
Therefore we have an isomorphism of translation quivers
$$\loc_\bd(\Ga)=\loc_{\bd^+}^\om(\Ga)\iso\loc_{\bd^+}(\Ga^+)^{\ta_\om}.$$


\begin{remark}
\label{rep as localiz}
In particular, for any quiver $Q$, consider the double quiver $\bar Q$ with its translation structure described in Example~\ref{double}.
Consider the weight map
$$\bd:\bar Q_1\to\bZ,\qquad a\mto
\begin{cases}
0&a\in Q_1,\\
1&a\in Q_1^*.
\end{cases}
$$
Then the localization $\loc_{\bd}(\bar Q)$ coincides with the repetition quiver $\bZ Q$ from Example \ref{repetition} and
$$\bZ Q=\loc_{\bd}(\bar Q)=\loc_{0}^1(\bar Q)\iso\loc_0(Q)^{\ta_1},\qquad \ta=\id.$$
\end{remark}

\begin{example}
\label{graded qv as tqv}
The following example appears in the study of graded (and cyclic) quiver varieties (see \eg~\cite{nakajima_quivere,nakajima_quiverf}).
Given a quiver $Q$, consider a new quiver \Ga with the set of vertices $Q_0\xx\bZ$ (or $Q_0\xx\bZ/r\bZ$ in the cyclic case) and with arrows
$$a_n:(i,n+1)\to(j,n),\qquad a_n^*:(j,n+1)\to (i,n)$$
for $a:i\to j$ in $Q_1$ and $n\in\bZ$.
Define
$$\ta(i,n)=(i,n+2),\qquad \si(a_n)=a_{n+1}^*,\qquad \si(a^*_n)=a_{n+1}.$$
We define the cut $\Ga^+_1\sbs\Ga_1$ consisting of arrows $a_n$, for $a\in Q_1$ and $n\in\bZ$.
This translation quiver can be realized as a localization quiver of $\bar Q$ \wrt the weight map 
$$\bd:\bar Q_1\to\bZ,\qquad a\mto-1.
$$
In the case of graded quiver varieties one studies representations of a framed quiver defined as follows.
Let $\bw\in\bN^{Q_0\xx\bZ}$ be a collection of non-negative integers.
We define the \idef{framed quiver} ~$\Ga^\f$ by adding to the quiver \Ga one new vertex $*$ as well as 
$\bbw_{i,n}$ arrows $*\to(i,n-1)$ and
$\bbw_{i,n}$ arrows $(i,n+1)\to*$, for all $i\in Q_0$ and $n\in\bZ$.
We can construct a bijection \si between the set of arrows $*\to (i,n-1)$ and the set of arrows $\ta(i,n-1)=(i,n+1)\to *$.
Note that in general the numbers of arrows $(i,n)\to*$ and $*\to(i,n)$ are different.
Therefore we obtain only a partial translation quiver $(\Ga^\f,\ta,\si)$ with the domain of $\ta$ equal $\Ga_0\sbs\Ga^\f_0$.
Note that $\Ga^\f$ admits a partial cut consisting of arrows $a_n$ and arrows $*\to(i,n)$.
In the next section we will discuss how one can extend $\Ga^\f$ to a (stable) translation quiver.
\end{example}

\sec[Stabilization]
\label{stab2}
Let $\Ga$ be a partial translation quiver such that $\ta\Ga'_0=\Ga'_0$.
Our goal is to construct a (stable) translation quiver $\hat\Ga$ such that $\Ga$ is its full subquiver with a compatible translation.
Our construction will be different from the construction in Remark \ref{stab1} as we will propagate only vertices outside of $\Ga'_0$.  
Let $S_0=\Ga_0\ms\Ga'_0$ and let $\Ga'\sbs \Ga$, $S\sbs\Ga$ be the full subquivers with the sets of vertices $\Ga'_0$, $S_0$ respectively.
We will construct a translation quiver $\hat\Ga$ by adding arrows to the quiver $\Ga'\sqcup\bZ S$, where $\bZ S$ is the repetition quiver of $S$.
Note that $\Ga'$ and $\bZ S$ are both translation quivers.
Denote a vertex $(s,n)\in \bZ S_0$ by $s[n]$.
We consider $S$ as a full subquiver of $\bZ S$ by identifying $s\in S_0$ with $s[0]\in\bZ S_0$.
By the definition of $\bZ S$ we have $\ta s[n]=s[n-1]$ and
$$\hat\Ga(s[n],t[n])=\hat\Ga(t[n-1],s[n])=\Ga(s,t),\qquad s,t\in S_0,\ n\in\bZ.$$
Define new arrows in $\hat\Ga$ by
$$\hat\Ga(s[n],i)=\hat\Ga(i,s[n+1])=\Ga(s,\ta^n i),
\qquad s\in S_0,\ i\in\Ga'_0,\ n\in\bZ.$$
Note that 
$$\hat\Ga(s[0],i)=\Ga(s,i),\qquad \hat\Ga(i,s[0])=\Ga(s,\ta\inv i)\iso\Ga(i,s),$$
hence we can consider \Ga as a full subquiver of $\hat\Ga$.
We extend $\si$ to $\hat\Ga$ using the identifications
$$\hat\Ga(s[n],i)
=\Ga(s,\ta^{n-1}\ta i)
=\hat \Ga(\ta i,s[n]),$$
$$\hat\Ga(i,s[n])
=\hat\Ga(s[n-1],i)
=\hat\Ga(\ta s[n],i).$$
This makes $\hat\Ga$ a translation quiver.


\subsection{Framed quivers and their stabilization}
\label{sec:framed}
Let $\Ga$ be a translation quiver and let $\bw\in\bN^{\Ga_0}$.
We construct the framed quiver $\Ga^\f$ by adding to $\Ga$ one new vertex $*$ as well as $\bbw_i$ arrows $*\to i$ and $\bbw_i$ arrows $\ta i\to*$ for all $i\in\Ga_0$.
As before, we obtain a partial translation quiver $(\Ga^\f,\ta,\si)$ with the domain of $\ta$ equal to $\Ga_0\sbs\Ga^\f_0$.
Applying the above stabilization procedure, we obtain a new translation quiver $\hGaf$ with the set of vertices $\Ga_0\sqcup\bZ$.
We will denote the vertex $*[n]$ by $[n]$, for $n\in\bZ$.
The arrows of $\hGaf$ are arrows from $\Ga$ as well as $\bbw_{\ta^n i}$ arrows $[n]\to i$ and $\bbw_{\ta^n i}$ arrows $\ta i\to [n]$, for all $i\in \Ga_0$ and $n\in\bZ$.
The translation extends from $\Ga$ to $\hGaf$ by $\ta[n]=[n-1]$ for $n\in\bZ$.
We call $\hGaf$ the \idef{stable framed quiver}. 

\begin{remark}
If $\Ga$ admits a cut $\Ga^+_1\sbs\Ga_1$, then $\Ga^\f$ admits a partial cut consisting of arrows in $\Ga^+_1$ and arrows $*\to i$ for $i\in\Ga_0$.
Similarly, $\hGaf$ 
admits a cut consisting of arrows in $\Ga^+_1$ and arrows
$[n]\to i$ for $i\in\Ga_0$ and $n\in\bZ$.
\end{remark}

\begin{example}
Let $\Ga$ be a translation quiver with vertices $x,y$ and no arrows, and with $\ta(x)=y$, $\ta(y)=x$.
Let $\bw=(0,1)\in\bN^{\Ga_0}$.
Then the framed quiver $\Ga^\f$ has the form
$x\xto b *\xto a y$.
It is a partial translation quiver with $\si(a)=b$.
Its stabilization $\hGaf$ is the quiver
\begin{ctikzcd}[column sep=1.7cm]
&&&x\dar["\si a"]\ar[drr,"\si^{-3}a"]\ar[dll,"\si^5a"',close]\\
\ar[r,-,dotted]& {[-2]}\ar[drr,"\si^4a"',close]&{[-1]}\ar[ur,"\si^2a"',close]&
{[0]}\dar["a"]&{[1]}\ar[ul,"\si^{-2}a",close]&
{[2]}\ar[dll,"\si^{-4}a",close]&\ar[l,-,dotted] \\
&&&y\ar[ul,"\si^3 a"']\ar[ur,"\si\inv a", close]
\end{ctikzcd}

The only \si-orbit path $\dots \to\ta j \xto{\si a}i\xto aj\xto{\si\inv a}\ta\inv i\to{}\dots$
has the form
$$\dots\to[-1]\to x\to [0]\to y\to [1]\to x\to[2]\to[y]\to\dots$$
\end{example}


\subsection{Mesh algebra}
\label{sec:mesh}
Let $\Ga$ be a translation quiver with a cut $\Ga^+_1\sbs\Ga_1$.
We define 
$$\eps:\Ga_1\to\bZ,\qquad 
\eps(a)=\begin{cases}
1&a\in \Ga^+_1,\\
-1&a\in \si\Ga^+_1.
\end{cases}$$
Let $k\Ga$ be the path algebra of $\Ga$ over a field $k$.
We define the \idef{mesh algebra} (or twisted pre-projective algebra)
\begin{equation}
\label{mesh alg}
\Pi(\Ga)=\Pi(\Ga,\ta)=k\Ga/(\fr),
\end{equation}
where the \idef{mesh relation} $\fr$ is given by
\begin{equation}
\label{mesh rel}
\fr=\sum_{a\in\Ga_1}\eps(a)a\cdot \si(a)\in k\Ga.
\end{equation}

Note that the translation $\ta$ induces an algebra automorphism $\ta:k\Ga\to k\Ga$ and we have $\ta(\fr)=\fr$.
Therefore $\ta$ induces an automorphism $\ta:\Pi(\Ga)\to\Pi(\Ga)$.
If $\Ga=\bar Q$ is the double quiver of a quiver $Q$, then $\Pi(\Ga)$ is the pre-projective algebra $\Pi_Q$ of $Q$ (see \eg~\cite{crawley-boevey_noncommutative}).

\medskip
Let us show that the mesh algebra is independent of the cut.
Assume that we have another cut $Q_1\sbs\Ga_1$ and let $\eps':\Ga_1\to\bZ$ be defined by $\eps'(a)=1$ if $a\in Q_1$ and $\eps'(a)=-1$ otherwise.
Define 
$$\eta:\Ga_1\to\bZ,\qquad a\mto
\begin{cases}
-1& a\in\Ga_1^+\ms Q_1,\\
1&\text{otherwise}
\end{cases}$$
and define the algebra automorphism
$$\vi:k\Ga\to k\Ga,\qquad \Ga_1\ni a\mto\eta(a)a.$$
If $a$ is contained in $\Ga_1^+\cap Q_1$ or
$\si(\Ga_1^+\cap Q_1)$, then 
$\eta(a)\eta(\si a)=1=\eps(a)\eps'(a)$.
Otherwise $\eta(a)\eta(\si a)=-1=\eps(a)\eps'(a)$.
This implies that
$$\vi\rbr{\sum_{a\in\Ga_1}\eps(a)a\si(a)}=\sum_{a\in\Ga_1}\eps'(a)a\si(a)$$
and $\vi$ induces an isomorphism of the corresponding mesh algebras.

\begin{remark}
Let $\Ga$ be a translation quiver which is a bipartite graph,
meaning that there is a decomposition $\Ga_0=I\sqcup J$ such that all arrows connect vertices in $I$ with vertices in $J$.
Consider the cut $\Ga_1^+\sbs\Ga_1$ consisting of all arrows from $I$ to $J$.
Then
$$
e_{j}\fr e_{\ta j}=
\begin{cases}
\sum_{t(a)=j}a\si(a),&j\in J,\\
-\sum_{t(a)=j}a\si(a),&j\in I.
\end{cases}
$$
This implies that $(\fr)=(\fr')$ and $\Pi(\Ga)=k\Ga/(\fr)=k\Ga/(\fr')$
for the relation
$$\fr'=\sum_{a\in\Ga_1}a\si(a).$$
\end{remark}

\begin{remark}
Let $\Ga=\bZ Q$ be the repetition quiver of a quiver $Q$
from Example \ref{repetition}.
Then it is common to use the mesh relation \rdt
$$\fr'=\sum_{a\in\Ga_1}a\si(a).$$
The corresponding quotient algebras are isomorphic.
Indeed, consider the isomorphism of algebras $\vi:k\Ga\to k\Ga$ given by
$$\vi(a_n)=(-1)^n a_n,\qquad \vi(a_n^*)=(-1)^n a_n^*.$$
We have $\si(a_n)=a^*_n$ and $\si(a^*_n)=a_{n-1}$, hence
$$\vi(a_n\si(a_n))=a_n\si(a_n),\qquad \vi(a_n^*\si(a_n^*))=-a_n^*\si(a_n^*).$$
This implies that $\vi(\fr)=\fr'$, hence $\vi$ induces an isomorphism $k\Ga/(\fr)\iso k\Ga/(\fr')$.
\end{remark}

\begin{remark}
More generally, assume that $\Ga$ is a translation quiver such that every \si-orbit is either infinite or has the number of elements divisible by $4$.
Let us fix a cut $\Ga^+_1\sbs\Ga_1$.
Then there exists a map
$\eta:\Ga_1^+\to\set{\pm1}$ such that $\eta(a)+\eta(\ta a)=0$, for all $a\in \Ga^+_1$.
We extend it to $\eta:\Ga_1\to\set{\pm1}$ by $\eta(\si a)=\eta(a)$, for all $a\in \Ga^+_1$.
The algebra automorphism
$$\vi:k\Ga\to k\Ga,\qquad \Ga_1\ni a\mto \eta(a) a$$
satisfies
$$\vi(a\cdot \si a-\si a\cdot \si^2 a)=a\cdot \si(a)+\si a\cdot \si^2a,\qquad a\in \Ga^+_1.$$
This implies that $\vi(\fr)=\fr'$ and $\vi$
induces an isomorphism $k\Ga/(\fr)\iso k\Ga/(\fr')$.
\end{remark}

Let $\Ga$ be a partial translation quiver with a partial cut $\Ga^+_1\sbs\Ga_1$ from Remark \ref{partial cut}.
Then the set $\Ga''_1=\sets{a:i\to j}{j\in\Ga'_0}$ is contained in $\Ga^+_1\sqcup\si\Ga^+_1$ and we can define $\eps:\Ga''_1\to\bZ$ in the same way as before.
We define the mesh relation and the mesh algebra $$\fr=\sum_{a\in\Ga''_1}\eps(a)a\si(a),\qquad
\Pi(\Ga)=\Pi(\Ga,\ta)=k\Ga/(\fr).$$ 

\sec[Coverings]
\label{coverings}
A quiver morphism $\pi:\tl Q\to Q$ is called a \idef{covering} if it is surjective on vertices and,
for every vertex $i\in \tl Q_0$, the map $\pi$ induces a bijection between the set of all arrows outgoing from $i$ and the set of all arrows outgoing from $\pi(i)$, and the same is true for ingoing arrows.
This implies that for any $i\in \tl Q_0$ and a path $u$ in $Q$ that starts at $\pi(i)$, there exists a unique path $\tl u$ that starts at $i$ such that $\pi(\tl u)=u$.

Let $G$ be a group acting on a quiver $\tl Q$,
meaning that we have a group homomorphism $\rho:G\to\Aut(\tl Q)$.
For any $g\in G$, we will usually denote $\rho(g)$ as $g$ by abuse of notation.
We will say that the action of $G$ on $\tl Q$ is 
\idef{admissible} if whenever $\tl Q(i,j)\ne\es$, the stabilizers $G_i$ and $G_j$ are equal and the action of $G_i$ on $\tl Q(i,j)$
is trivial.
Note that the usual definition of an admissible action is more restrictive \cite{martinez-villa_universal}.

Given an action of a group $G$ on a quiver $\tl Q$,
we define a new quiver $Q=\tl Q/G$ with $Q_0=\tl Q_0/G$ and $Q_1=\tl Q_1/G$.
Let $\pi:\tl Q\to Q=\tl Q/G$ be the corresponding projection.
If the action is admissible, then
$$Q([i],[j])=\bigsqcup_{g\in G/G_i}\tl Q(i,gj)$$
and the projection $\pi:\tl Q\to Q$ is a covering.
Note that if $\tl Q$ is locally finite (every vertex is incident with finitely many arrows), then so is $Q$.

Given a covering $\pi:\tl Q\to Q$,
we construct a functor
$$\pi_*:\Rep(\tl Q)\to \Rep(Q),\qquad
\tl M\mto M,\qquad
M_i=\bop_{\pi(k)=i}\tl M_k,\qquad M_a=\sum_{\pi(\tl a)=a}M_{\tl a},$$
for $i\in Q_0$ and $a\in Q_1$.

Similarly, we say that a translation quiver morphism $\pi:\tl\Ga\to\Ga$ is a covering if it is a covering of quivers.
If $\Ga^+_1\sbs\Ga_1$ is a cut of $\Ga$, then $\tl\Ga^+_1=\pi\inv(\Ga^+_1)$ is a cut of $\tl\Ga$.
In this case the functor $\pi_*:\Rep(\tl \Ga)\to \Rep(\Ga)$ induces the functor
$$\pi_*:\mmod\Pi(\tl\Ga)\to\mmod\Pi(\Ga).$$
Let $G$ be a group acting on a translation quiver $\tl\Ga$ by translation quiver automorphisms, meaning that for every $g\in G$, the corresponding automorphism $g:\tl\Ga\to\tl\Ga$ preserves $\ta$ and ~$\si$.
If the action of $G$ is admissible, then $\Ga=\tl\Ga/G$ inherits a translation quiver structure and $\pi:\tl\Ga\to\Ga$ is a covering of translation quivers.

In particular, let $\Ga$ be a translation quiver, $\bd:\Ga_1\to\La$ be a weight map and $\tl\Ga=\loc_\bd(\Ga)$ be the corresponding localization quiver.
Then the group $\La$ acts on the translation quiver~$\tl\Ga$
$$m\circ (i,n)=(i,n+m),\qquad m\circ a_n=a_{n+m},\qquad
i\in\Ga_0,\, a\in\Ga_1,\, m,n\in\La.$$
This action is admissible and $\tl\Ga/\La$ is isomorphic to $\Ga$.
This implies that the projection map $\pi:\tl\Ga\to\Ga$ is a covering.
If $\Ga$ has a cut $\Ga_1^+\sbs\Ga_1$, then the cut $\tl\Ga_1^+=\pi\inv(\Ga_1^+)$ of $\tl\Ga$ is exactly the cut
defined in Example \ref{localiz}.

\begin{example}[\cf~\cite{scherotzke_generalized}]
Let $\tl\Ga$ be a locally finite translation quiver, meaning that every vertex is incident with finitely many arrows.
Let $\nu:\tl\Ga\to\tl\Ga$ be a translation quiver automorphism
such that $\nu^n(i)\ne i$ for all $i\in\tl\Ga_0$ and $n\ge1$.
Then the action of $\bZ$ on $\tl\Ga$ given by
$$n\circ i=\nu^n(i),\qquad n\circ a=\nu^n(a),\qquad i\in\tl\Ga_0,\, a\in\tl\Ga_1,\, n\in\bZ,$$
is free on vertices, hence is admissible.
The quotient quiver $\Ga=\tl\Ga/\nu=\tl\Ga/\bZ$ has the set of vertices consisting of $\nu$-orbits in $\tl\Ga_0$ and the sets of arrows
$$\Ga([i],[j])=\bigsqcup_{n\in\bZ}\tl\Ga(i,\nu^nj).$$
As before, $\Ga$ is a translation quiver, with the translation given by $\ta[i]=[\ta i]$.
\end{example}

%% file: homological_prop.tex
\section{Homological properties}

\subsection{Quiver representations}
Let $Q$ be a quiver.
We don't require $Q_0$ to be finite, but we require that the number of arrows between any two vertices is finite.
Given a quiver~ $Q$, we define its \idef{representation} $M$ to be a collection of vector spaces $(M_i)_{i\in Q_0}$ together with a collection 
of linear maps $M_a:M_i\to M_j$ for all arrows $a:i\to j$ in $Q$.
We will consider only finite-dimensional representations, meaning that $\sum_i\dim M_i<\infty$.
Given a path $u=a_n\dots a_1$ in $Q$,
we define $u|M=M_u=M_{a_n}\dots M_{a_1}$ considered as an endomorphism of $M=\bop_i M_i$.
We extend this definition to the elements of the path algebra $kQ$ by linearity.
Let $A=kQ/I$ be the quotient algebra by some ideal $I$.
We can identify $A$-modules with $Q$-representations $M$ that vanish on ~$I$, meaning that $M_u=0$ for all $u\in I$.
We will call them also \idef{$A$-representations}.

Given an abelian monoid \La and a set $X$, we define
$$\ds\La X=\sets{f\in\La^X}{\#\supp f<\infty},
\qquad \supp f=\sets{x\in X}{f_x\ne0}.$$
Given a $Q$-representation $M$, we define its 
\idef{dimension vector} $\udim M=(\dim M_i)_{i\in Q_0}\in\bN^{(Q_0)}$.
Define the Euler-Ringel form of $Q$
$$\hi(m,n)=\sum_i m_in_i-\sum_{a:i\to j}m_in_j,\qquad m,n\in\bZ^{(Q_0)}.$$
Then, for any two $Q$-representations $M,N$, we have
$$\hi(M,N)=\dim\Hom(M,N)-\dim\Ext^1(M,N)
=\hi(\udim M,\udim N).$$

Given a quiver automorphism $\ta:Q\to Q$ and a $Q$-representation $M$, we define a new $Q$-representation
$M^\ta$ with 
\begin{equation}\label{ta-shifted rep}
M^\ta_i=M_{\ta i}\quad \forall i\in Q_0,\qquad
M^\ta_a=M_{\ta a}:M_{\ta i}\to M_{\ta j}\quad\forall
(a:i\to j)\in Q_1.
\end{equation}
Similarly, for any $m\in\bZ^{Q_0}$, define $m^\ta=(m_{\ta i})_{i\in Q_0}$.
If $I\sbs kQ$ is a \ta-invariant ideal, then $A=kQ/I$ inherits the action by $\ta$.
For any $A$-representation $M$, we obtain an $A$-representation~$M^\ta$.

\subsection{Lift properties}
\label{sec:lift}
Let $(\Ga,\ta,\si)$ be a translation quiver with a cut $\Ga^+_1$ (see \S\ref{cuts}) and let $\Pi=\Pi(\Ga)$ be the corresponding mesh algebra (see \S\ref{sec:mesh}).

\begin{proposition}[\cf \cite{crawley-boevey_noncommutative}]
Let $M$ be a $\Ga^+$-representation.
Then there is an exact sequence
\begin{multline}
0\to D\Ext^1(M,M^\ta)
\to \bop_{(a:i\to j)\in\Ga_1^+}\Hom(M_{\ta j},M_{i})
\xto\vi\bop_i\Hom(M_{\ta i},M_{i})\\
\to D\Hom(M,M^\ta)\to0
\end{multline}
where $\vi$ sends $(M_{\si a})_{a\in\Ga^+_1}$ to 
$\sum_{a\in\Ga_1}\eps(a)M_a M_{\si a}$.
This implies that we can identify the space $D\Ext^1(M,M^\ta)$ with the set of lifts of the $\Ga^+$-representation $M$ to a $\Pi$-representation.
\end{proposition}
\begin{proof}
Given two representations $M,N$ of $\Ga^+$, there is an exact sequence
$$0\to\Hom(M,N)\to\bop_i\Hom(M_i,N_i)\to\bop_{a:i\to j}\Hom(M_i,N_j)\to\Ext^1(M,N)\to0$$
In particular, for the representations $M$ and $M^\ta$, we have
$$0\to\Hom(M,M^\ta)\to\bop_i\Hom(M_{i},M_{\ta i})\to\bop_{a:i\to j}\Hom(M_{i},M_{\ta j})\to\Ext^1(M,M^\ta)\to0$$
Dualizing, we obtain the result.
\end{proof}

\subsection{Quadratic algebras}
The mesh algebra $\Pi=\Pi(\Ga)$ is an example of a quadratic algebra.
In this section we will discuss such algebras and some of their properties in more detail.

Let $S=\prod_{i\in I}k$ be a finite-dimensional semisimple algebra over a field $k$.
Given $S$-bimodules $V$ and $W$, we consider 
$V\ts W:=V\ts_SW$ as an $S$-bimodule.
Let $V$ be a finite-dimensional $S$-bimodule 
and $R\sbs V\ts V$ be a sub-bimodule.
We define the free tensor algebra 
$$T_SV=S\oplus V\oplus (V\ts V)\oplus\dots=\bop_{n\ge0}V^{\ts n}$$
and define the \idef{quadratic algebra} $A=T_SV/(R)$, where $(R)\sbs T_SV$ is the ideal generated by $R$.
The algebra $A$ is graded, with $A=\bop_{n\ge0}A_n$ and $A_n$ equal to the image of $V^{\ts n}$ in ~$A$.
Note that $A_0=S$ and $A_1=V$.

We define the \idef{Koszul complex} (see \eg \cite[\S2.6]{beilinson_koszul})
$$\dots\to K^2\to K^1\to K^0\to0\to\dots$$
consisting of projective graded (left) $A$-modules as follows.
Let $K_n^n\sbs V^{\ts n}$ be defined by
\begin{gather*}
K_0^0=S,\qquad K_1^1=V,\qquad K_2^2=R,\\
K_{n+1}^{n+1}=(V\ts K_n^n)\cap(K_n^n\ts V)\sbs V^{\ts (n+1)},
\qquad n\ge2.
\end{gather*}
Then we define $K^n=A\ts_S K_n^n$ (where $K^n_n$ is the degree $n$ component of $K^n$) and define the differential
$$d:K^n\to K^{n-1},\qquad a\ts v_1\ts\dots\ts v_n\mto av_1\ts \dots\ts v_n.$$
The algebra $A$ is Koszul if and only if the Koszul complex is a resolution of $S=A/\bar A$ considered as an $A$-module, where $\bar A=\bop_{n>0}A_n$ is the augmentation ideal
\cite[\S2.6]{beilinson_koszul}.
For an arbitrary quadratic algebra ~$A$,
the first terms of the Koszul complex give a resolution of $S$.
More precisely, we have an exact sequence
(see \eg \cite[Lemma 3.8]{brenner_periodic})
\begin{equation}
\label{Kos exact}
A\ts_SR\to A\ts_S V\to A\to S\to0,
\end{equation}
meaning that the sequence is exact in the three middle terms.

Let us define the \idef{Koszul bimodule complex}
$$\dots\to \wtl K^2\to \wtl K^1\to \wtl K^0\to0\to\dots,$$
where $\wtl K^n=A\ts_S K_n^n\ts_S A$ and the differential 
$d:\wtl K^n\to \wtl K^{n-1}$ is given by
$$a\ts v_1\ts\dots\ts v_n\ts b\mto 
av_1\ts \dots\ts v_n\ts b
+(-1)^n a\ts v_1\ts \dots\ts v_nb
.$$
Note that $K^\bul=\wtl K^\bul\ts_AS$.
The quadratic algebra $A$ is Koszul if and only if the Koszul bimodule complex is a resolution of $A$ as an $A$-bimodule
\cite[A.2]{braverman_poincare}.
For an arbitrary quadratic algebra $A$, the first terms of the Koszul bimodule complex give a resolution of ~$A$ as an $A$-bimodule.

\begin{lemma}\label{Kos bimod exact}
Let $A=T_SV/(R)$ be a quadratic algebra.
Then we have an exact sequence
\begin{equation}
A\ts_SR\ts_S A\to A\ts_S V\ts_S A\to A\ts_S A\to A\to0.
\end{equation}
\end{lemma}
\begin{proof}
According to \cite[A.3]{braverman_poincare}, a complex $\wtl P^\bul$ of free $\bN$-graded right $A$-modules is exact if $\wtl P^\bul\ts_AS$ is exact.
Now we apply the fact that \eqref{Kos exact} is exact.
\end{proof}

\subsection{Special exact sequence}
Let us consider now a translation quiver $(\Ga,\ta,\si)$ with a cut $\Ga^+_1$ and let $\Pi=\Pi(\Ga)$ be the corresponding mesh algebra.
Let us define the algebra
$S=\bop_{i\in \Ga_0}ke_i$ and 
the $S$-bimodule $V=\bop_{a\in\Ga_1}ka$.
Then the tensor algebra $T_SV$ is isomorphic to the path algebra $k\Ga$.
The mesh algebra $\Pi=\Pi(\Ga)$ is equal to $k\Ga/(\fr)$,
where $\fr$ is given by \eqref{mesh rel}
\begin{equation}
\fr=\sum_{a\in\Ga_1}\eps(a)a\cdot \si(a),
\qquad
\eps(a)=\begin{cases}
1&a\in \Ga^+_1,\\
-1&a\in \si\Ga^+_1.
\end{cases}
\end{equation}
Note that for any arrow $a:i\to j$ the path $a\cdot \si(a)$ goes from $\ta j$ to $j$.
For any vertex $j\in\Ga_0$, we can consider
$\fr_j=e_{j}\fr e_{\ta j}$
as an element of $V\ts V$.
Then $k\fr_j\sbs V\ts V$ is an $S$-sub-bimodule and we define
\begin{equation}
R=\bop_{j\in\Ga_0}k\fr_j\sbs V\ts V,\qquad
\fr_j=\sum_{t(a)=j}\eps(a) a\ts\si (a).
\end{equation}
We conclude that the mesh algebra is equal to the quadratic algebra
\begin{equation}
\Pi=\Pi(\Ga)=T_SV/(R).
\end{equation}
We obtain from Lemma \ref{Kos bimod exact} an exact sequence of $\Pi$-bimodules
\begin{equation}\label{proj res}
\bop_{i\in\Ga_0} \Pi e_i\ts e_{\ta i}\Pi\to
\bop_{a\rcol i\to j}\Pi e_j\ts e_i\Pi\to
\bop_{i\in \Ga_0} \Pi e_i\ts e_i\Pi\to \Pi\to0.
\end{equation}


\begin{proposition}[\cf \cite{crawley-boevey_exceptional}]
\label{prop:ex seq}
Given two $\Pi$-representations $M,N$, there is a complex
$$\dots\to0\to \bop_{i\in \Ga_0}\Hom(M_i,N_i)
\to\bop_{(a\rcol i\to j)\in \Ga_1}\Hom(M_i,N_j)
\to \bop_{i\in \Ga_0}\Hom(M_{\ta i},N_{i})\to0\to \dots$$
whose cohomology groups are $\Hom_\Pi(M,N)$, $\Ext^1_\Pi(M,N)$ and $D\Hom_\Pi(N,M^\ta)$.
\end{proposition}
\begin{proof}
Applying $-\ts_\Pi M$ 
to the exact sequence \eqref{proj res},
we obtain a projective resolution
$$
\dots\to\bop_{i\in\Ga_0} \Pi e_i\ts M_{\ta i}\xto f
\bop_{a\rcol i\to j}\Pi e_j\ts M_i\xto g
\bop_{i\in\Ga_0} \Pi e_i\ts M_i\to M\to0$$
where, using $p_i\in \Pi e_i$ and $m_i\in M_i$, we define
\begin{gather*}
f\rbr{\sum p_i\ts m_{\ta i}}=
\sum_{a\rcol i\to j}
\eps(a)\Big(
p_ja\ts m_{\ta j}
-p_{i}\ts \ta a\cdot m_{\ta i}\Big)_{\si a}\\
g(p_j\ts m_i)_a=(p_j a\ts m_i)_i-(p_j\ts am_i)_j,\qquad a:i\to j.
\end{gather*}


Applying $\Hom(-,N)$ and identifying $\Hom(\Pi e_j\ts M_i,N)$ with $\Hom(M_i,N_j)$, we obtain the complex
\begin{gather*}
0\to \bop_{i\in \Ga_0}\Hom(M_i,N_i)
\xto {\bar g}\bop_{a\rcol i\to j}\Hom(M_i,N_j)
\xto {\bar f}\bop_{i\in \Ga_0}\Hom(M_{\ta i},N_{i})\to\dots\\
\bar g(\psi)=\sum_{a\rcol i\to j}(a\psi_i-\psi_j a)_a,
\qquad \psi=(\psi_i:M_i\to N_i)_i,\\
\bar f(\vi)
=\sum_{a\rcol i\to j}\eps(a)(a\vi_{\si a}+\vi_{a}\cdot \si a)_j,\qquad
\vi=(\vi_a:M_i\to N_j)_{a\rcol i\to j}.
\end{gather*}

The first two cohomology groups are $\Hom_\Pi(M,N)$ and $\Ext^1_\Pi(M,N)$.
Dualizing, we see that the cokernel of $\bar f$ is $D\Hom_\Pi(N,M^\ta)$.
\end{proof}

For any two $\Pi$-modules $M,N$, we define $h^i(M,N)=\dim\Ext^i_\Pi(M,N)$.

\begin{proposition}
\label{h^i}
For any $\Pi$-modules $M,N$ we have
$$h^0(M,N)-h^1(M,N)+h^0(N,M^\ta)=\hi(M,N)+\hi(N,M^\ta),$$
where $\hi$ is the Euler-Ringel form of the quiver $\Ga^+$.
\end{proposition}
\begin{proof}
If $m=\udim M$ and $n=\udim N$, then the Euler characteristic of the above complex is
$$\sum m_in_i-\sum_{(a:i\to j)\in\Ga^+_1}(m_in_j+m_{\ta j} n_i)+\sum_i m_{\ta i}n_i
=\hi(M,N)+\hi(N,M^\ta).
$$
\end{proof}

%% file: quiver_var.tex
\section{Translation quiver varieties}

\subsection{Quiver varieties}
\label{sec:QV}
The material of this section is well-known \cite{king_moduli}.
We include it to fix notation.
Let $Q$ be a quiver (for simplicity we assume it to be finite)
and let $k$ be an algebraically closed field.
Let $I\sbs kQ$ be an ideal of the path algebra contained in the ideal $J\sbs kQ$ generated by all arrows.
Let $A=kQ/I$ be the quotient algebra and $\bv\in\bN^{Q_0}$.
We are going to introduce the moduli space of semistable $A$-modules having dimension vector $\bv$.


Let $V$ be a $Q_0$-graded vector space having the dimension vector $\bv\in\bN^{Q_0}$.
We define the representation space
$$\cR(Q,\bv)=\bop_{a:i\to j}\Hom(V_i,V_j)$$
equipped with the action of the group
$\GL_\bv=\prod_i\GL(V_i)$
given by $(g\cdot M)_a=g_jM_a g_i\inv$ for any arrow $a:i\to j$ in $Q$.
For the algebra $A=kQ/I$, let
$$\cR(A,\bv)\sbs\cR(Q,\bv)$$
be the closed subvariety consisting of representations that vanish on $I$. 
It is also equipped with an action of $\GL_\bv$.

Consider some $\te\in\bR^{Q_0}$ which we will call a stability parameter.
For any $\bv\in\bN^{Q_0}\ms\set0$, define the slope
$$\mu_\te(\bv)=\frac{\sum_i \te_i \bv_i}{\sum_i \bv_i}
=\frac{\te\cdot\bv}{\rho\cdot\bv},$$
where $\rho\in\bZ^{Q_0}$ is given by $\rho_i=1$ for all $i\in Q_0$.
For any nonzero $A$-module $M$, define $\mu_\te(M)=\mu_\te(\udim M)$.
It is called $\te$-semistable if
$$\mu_\te(N)\le\mu_\te(M)$$
for every submodule $0\ne N\subsetneq M$.
It is called \te-stable if the above inequalities are strict.
There exists an open subvariety $\cR_\te(A,\bv)\sbs\cR(A,\bv)$ consisting of $\te$-semistable submodules.
It is proved in \cite{king_moduli} that there exists a quasi-projective categorical quotient
$$\cM_\te(A,\bv)=\cR_\te(A,\bv)\GIT\GL_\bv$$
which parametrizes $S$-equivalence classes of $\te$-semistable $A$-modules having dimension vector $\bv$.
Here two \te-semistable $A$-modules are called $S$-equivalent if they have the same composition factors in the category of \te-semistable $A$-modules (of the same slope).

\begin{remark}
To be more precise, the construction in \cite{king_moduli} is formulated for $\te\in\bZ^{Q_0}$ such that $\te\cdot \bv=0$.
We can reduce semistability \wrt an arbitrary $\te\in\bR^{Q_0}$ to this case as follows.
First, we consider $\te'=\te-\mu_\te(\bv)\rho$ so that $\te'\cdot\bv=0$.
Semistability conditions \wrt $\te$ and $\te'$ are equivalent.
Next, we approximate $\te'$ by some $\te''\in\bQ^{Q_0}$ such that $\sgn(\te'\cdot \bu)=\sgn(\te''\cdot \bu)$ for all $0\le\bu\le\bv$.
Then semistability conditions \wrt $\te'$ and $\te''$ on representations having dimension $\bv$ are equivalent.
Finally, we can find an integer $k\ge1$ such that $k\te''\in\bZ^{Q_0}$.
\end{remark}

There also exists an open subvariety $\cR^\s_\te(A,\bv)\sbs\cR_\te(A,\bv)$ consisting of stable $A$-modules and a geometric quotient $$\cM^\s_\te(A,\bv)=\cR^\s_\te(A,\bv)/\GL_\bv$$
which parametrizes isomorphism classes of $\te$-stable $A$-modules having dimension vector~$\bv$.
The moduli space $\cM^\s_\te(A,\bv)$ is open in $\cM_\te(A,\bv)$.
We will say that $\te$ is \bv-generic if $\mu_\te(\bu)\ne\mu_\te(\bv)$ for all $0<\bu<\bv$ (meaning that $0\le \bu_i\le \bv_i$ for all $i\in Q_0$ and $0\ne \bu\ne \bv$).
In this case
$\cM^\s_\te(A,\bv)=\cM_\te(A,\bv)$.

\begin{remark}
\label{projection to semisimp}
In the case of the trivial stability parameter $\te=0$ all modules are semistable and a module is stable if and only if it is simple.
This implies that the moduli space $\cM_0(A,\bv)$ parametrizes isomorphism classes of semisimple $A$-modules having dimension vector \bv.
It can be described as
$$\cM_0(A,\bv)=\cR(A,\bv)\GIT\GL_\bv
=\Spec k[\cR(A,\bv)]^{\GL_\bv}.$$
For any $\te\in\bR^{Q_0}$, there exists a canonical projective morphism $\pi:\cM_\te(A,\bv)\to\cM_0(A,\bv)$
giving rise to a commutative diagram \cite{king_moduli}
\begin{ctikzcd}
\cR_\te(A,\bv)\rar[hook]\dar&\cR(A,\bv)\dar\\
\cM_\te(A,\bv)\rar["\pi"]&\cM_0(A,\bv)
\end{ctikzcd}
\end{remark}

We say that a module $M\in\mmod A$ is nilpotent if $J^nM=0$ for some $n\ge1$, where $J$ is the ideal generated by all arrows.
Let $\cL_\te(A,\bv)\sbs\cM_\te(A,\bv)$ be the subvariety of nilpotent modules.

\begin{lemma}
We have $\cL_\te(A,\bv)=\pi\inv(0)$.
In particular, $\cL_\te(A,\bv)$ is projective.
\end{lemma}
\begin{proof}
A module $M\in \cM_\te(A,\bv)$ satisfies $\pi(M)=0$ if and only if $M$ is $S$-equivalent to the direct sum $\bop_{i}S(i)^{\bv_i}$ of $1$-dimensional modules $S(i)$.
This means that $M$ has a filtration $0=M_0\sbs\dots\sbs M_n=M$ such that $J (M_{k}/M_{k-1})=0$, for all $k$.
Then $J^nM=0$.
The converse is similar.
\end{proof}

\subsection{Nakajima quiver varieties}
For an introduction to Nakajima quiver varieties see \eg~ \cite{ginzburg_lecturesa}.
Let $Q$ be a finite quiver, $\Ga=\bar Q$ be its double quiver
and $\Pi=\Pi(\Ga)$ be the mesh algebra.
Let $\bv\in\bN^{Q_0}$ and $\te\in\bR^{Q_0}$.
Then one defines Nakajima quiver variety to be 
$\cM_\te(\Pi,\bv)$.
There is an alternative approach that uses moment maps.
Consider the action of $\GL_\bv$ on $\cR(Q,\bv)$.
It induces a map
$$\rho:\gl_\bv=\prod_i\gl_{\bv_i}\to \cR(Q,\bv)^*\xx \cR(Q,\bv).$$
Dualizing, we obtain a map
$$\mu:\cR(\bar Q,\bv)=\cR(Q,\bv)\xx\cR(Q,\bv)^*\to\gl_\bv^*$$
called the moment map.
One can show that
$$\cR(\Pi,\bv)=\mu\inv(0).$$
This implies that $\cR_\te(\Pi,\bv)=\mu\inv(0)_\te=\mu\inv(0)\cap\cR_\te(Q,\bv)$ and
$$\cM_\te(\Pi,\bv)=\mu\inv(0)_\te\GIT\GL_\bv.$$

Let $\hi$ be the Euler-Ringel form of $Q$.
The following result is due to Nakajima
(\cf \cite[Theorem~2.8]{nakajima_instantons}),
although the original proof is different.

\begin{theorem}
Assume that $\te$ is $\bv$-generic.
Then $\cM_\te(\Pi,\bv)$ is smooth and has dimension $2-2\hi(\bv,\bv)$.
\end{theorem}
\begin{proof}
Given $\Pi$-modules $M,N$, let $h^i(M,N)=\dim\Ext^i_\Pi(M,N)$.
The tangent space of the quiver variety $\cM_\te(\Pi,\bv)=\cM_\te^\s(\Pi,\bv)$
at the point $M\in \cM_\te^\s(\Pi,\bv)$ can be identified with $\Ext^1_\Pi(M,M)$ and we need to show that $h^1(M,M)$ is independent of $M$.
By Corollary \ref{h^i} we have
$$2h^0(M,M)-h^1(M,M)=2\hi(M,M).$$
We have $h^0(M,M)=1$ as $M$ is stable.
Therefore $h^1(M,M)=2-2\hi(\bv,\bv)$.
\end{proof}


Given $\bw\in\bN^{Q_0}$ we constructed in \S\ref{sec:framed} the framed quiver $\Ga^\f=\bar Q^\f$ by adding to $\Ga=\bar Q$ one new vertex $*$ as well as $\bw_i$ arrows $*\to i$ and $\bw_i$ arrows $i\to *$, for all $i\in Q_0$.
It is the double quiver of the quiver $Q^\f$ obtained by adding to $Q$ one new vertex $*$ as well as $\bw_i$ arrows $*\to i$, for $i\in Q_0$.
We extend $\bv\in\bN^{Q_0}$ to $\bvf\in \bN^{Q_0^\f}$
by setting $\bvf_*=1$.
Define $\tef\in\bR^{Q_0^\f}$ with $\tef_i=0$ for $i\in Q_0$ and $\tef_*=1$.
A representation $M\in\cR(\Ga^\f,\bvf)$ is $\tef$-semistable if and only if for any $N\sbs M$ with $N_*\ne0$, we have $N=M$.
One defines the \idef{Nakajima quiver variety}
$$\cM(\bv,\bw)=\cM_{\tef}(\Pi(\Ga^\f),\bvf).$$
The stability parameter $\tef$ is $\bvf$-generic, hence by the previous theorem $\cM(\bv,\bw)$ is smooth and has dimension $2(\bv\cdot\bw-\hi(\bv,\bv))$.

\begin{remark}
Generally, we extend an arbitrary $\te\in\bR^{Q_0}$ 
with $\te\cdot\bv=0$ to $\tef\in\bR^{Q_0^\f}$ by setting $\tef_*=\eps$, for $0<\eps\ll1$.
This stability parameter is $\bvf$-generic and we define
$$\cM_\te(\bv,\bw)=\cM_{\tef}(\Pi(\Ga^\f),\bvf).$$
This quiver variety is smooth and has dimension
$2(\bv\cdot\bw-\hi(\bv,\bv))$.
\end{remark}

\subsection{Translation quiver varieties}
\label{translation qv}
Let $(\Ga,\ta,\si)$ be a translation quiver with a cut $\Ga^+_1$ and let $\Pi=\Pi(\Ga)$ be the mesh algebra.

\begin{definition}
Let $\bv\in\bN^{(\Ga_0)}$ and let $\te\in\bR^{\Ga_0}$ be $\bv$-generic. 
We define the \idef{translation quiver variety} to be $\cM_\te(\Pi,\bv)$.
\end{definition}

\begin{theorem}
\label{smoothness2}
Let $\te$ be $\bv$-generic, $\te^\ta=\te$ and $\bv^\ta\ne\bv$.
Then $\cM_\te(\Pi,\bv)$ is smooth and has dimension
$1-\hi(\bv,\bv+\bv^\ta)$,
where $\hi$ is the Euler-Ringel form of the quiver $\Ga^+$.
\end{theorem}
\begin{proof}
Given $\Pi$-modules $M,N$, let $h^i(M,N)=\dim\Ext^i_\Pi(M,N)$.
The tangent space of the quiver variety $\cM_\te(\Pi,\bv)=\cM_\te^\s(\Pi,\bv)$
at the point $M\in \cM_\te^\s(\Pi,\bv)$ can be identified with $\Ext^1_\Pi(M,M)$ and we need to show that $h^1(M,M)$ is independent of $M$.
By Proposition \ref{h^i}, we have
$$h^0(M,M)-h^1(M,M)+h^0(M,M^\ta)=\hi(M,M)+\hi(M,M^\ta).$$
The representation $M^\ta$ is stable with respect to $\te^\ta=\te$.
It is not isomorphic to $M$ as $\udim M^\ta=\bv^\ta\ne\bv$.
Therefore $h^0(M,M^\ta)=0$ and we obtain
$h^1(M,M)=1-\hi(\bv,\bv+\bv^\ta)$.
\end{proof}

Let $\bw\in\bN^{\Ga_0}$ and let $\Ga^\f$ be the corresponding framed quiver (see \S\ref{sec:framed}) which is a partial translation quiver equipped with a partial cut.
We extend $\bv\in\bN^{(\Ga_0)}$ to a dimension vector $\bvf$ over $\Ga^\f_0$ by setting $\bvf_*=1$.
Let us consider a stability parameter $\tef$ over $\Ga^\f$ with $\tef_i=0$ for $i\in\Ga_0$ and $\tef_*=1$.
Note that $\tef$ is $\bvf$-generic.

\begin{definition}
Given $\bv\in\bN^{(\Ga_0)}$ and $\bw\in\bN^{\Ga_0}$,
we define the \idef{(framed) translation quiver variety}
to be
$$\cM(\bv,\bw)=\cM_{\tef}(\Pi(\Ga^\f),\bvf),$$
where $\Pi(\Ga^\f)$ is the mesh algebra of the partial translation quiver $\Ga^\f$.
\end{definition}

Note that the map $M\mto M^\ta$ induces an isomorphism $\cM(\bv,\bw)\iso\cM(\bv^\ta,\bw^\ta)$.
In the next theorem we will see that $\cM(\bv,\bw)$ can be 
interpreted as a moduli space of representations of the mesh algebra $\Pi(\hGaf)$, where $\hGaf$ is the stabilization of $\Ga^\f$ (see \S\ref{sec:framed}).
This will allow us to apply the previous theorem which was proved for stable translation quivers.

\begin{theorem}
\label{framed smooth}
The translation quiver variety $\cM(\bv,\bw)$ is smooth and has dimension
$$\bw\cdot(\bv+\bv^\ta)-\hi(\bv,\bv+\bv^\ta),$$
where $\hi$ is the Euler-Ringel form of the cut $\Ga^+\sbs\Ga$.
\end{theorem}
\begin{proof}
Let $\hGaf$ be the stabilization of $\Ga^\f$ (see \S\ref{sec:framed}).
Let us extend $\bvf$ to a dimension vector over $\hGaf$ by zero.
Let us extend $\tef$ to a stability parameter over $\hGaf$
so that $\tef_{[n]}=1$, for all $n\in\bZ$. 
Then
$$\cM(\bv,\bw)=\cM_{\tef}(\Pi(\hGaf),\bvf).$$
The stability parameter $\tef$ is $\bvf$-generic and $(\tef)^\ta=\tef$.
On the other hand $(\bvf)^\ta\ne\bvf$ as $(\bvf)^\ta_{[0]}=\bvf_{[-1]}=0\ne\bvf_{[0]}=1$.
Therefore we can apply Theorem \ref{smoothness2}.
Let $Q$ be the cut of the translation quiver $\hGaf$ (its arrows are arrows in $\Ga^+$ and arrows $[n]\to i$ for $n\in\bZ$, $i\in\Ga_0$).
We will again use ~$\hi$ to denote the corresponding Euler-Ringel form.
Let $e_n$ be the dimension vector on $\hGaf$ that equals $1$ at the vertex $[n]$ and zero at all other vertices.
Considering $\bv$ as a dimension vector on $\hGaf$ we can write $\bvf=\bv+e_0$ and $\bvf^\ta=\bv^\ta+e_1$.
We have 
$$\hi(\bv+e_0,\bv+e_0)=\hi(\bv,\bv)-\bw\cdot\bv+1.$$
On the other hand $$\hi(\bv+e_0,\bv^\ta+e_1)=\hi(\bv,\bv^\ta)-\bw\cdot \bv^\ta$$
Therefore the dimension of the quiver variety is $$1-\hi(\bvf,\bvf+\bvf^\ta)
=\bw\cdot (\bv+\bv^\ta)-\hi(\bv,\bv+\bv^\ta).$$
\end{proof}

%% file: motivic_classes.tex
\section{Motivic classes}
In this section we will consider only algebraic varieties over $\bC$.

\sec[Motivic theory]
\label{motivic theory}
Let $\Var$ be the category of algebraic varieties over \bC,
with the monoidal category structure given by the Cartesian product.
Its identity object is $\pt=\Spec \bC$.
Let $\Ab$ be the category of abelian groups, with the monoidal category structure given by the tensor product.
The following structure is an example of a theory with transfer on the category $\Var$ with values in the category $\Ab$ (see \eg \cite{dyckerhoff_highera}).
We will call it a motivic theory (\cf \cite{davison_donaldson}).

For any $S\in\Var$, let $\Var\qt S$ be the category of algebraic varieties over $S$ and let $K(\Var\qt S)$ be the Grothendieck group of algebraic varieties over $S$ which is the free abelian group generated by isomorphism classes of objects in $\Var\qt S$ modulo relations
$$[X\to S]=[Y\to S]+[(X\ms Y)\to S],$$
for any variety $X$ over $S$ and a closed subvariety $Y\sbs X$.
We define $K(\Var)=K(\Var\qt\pt)$.
By applying resolutions of singularities, we can see that
$K(\Var)$ is generated by elements $[X]=[X\to\pt]$ for smooth, projective and connected varieties $X$.
We will write generators of $K(\Var\qt S)$ in the form $[X\to S]=[X]_S$.
The abelian groups $M(S)=K(\Var\qt S)$ have the following properties.
For any morphism $u:S\to T$, we have push-forward and pull-back morphisms
$$u_!:M(S)\to M(T),\qquad [X\to S]\mto[X\to S\xto u T],$$
$$u^*:M(T)\to M(S),\qquad [X\to T]\mto[X\xx_TS\to S].$$
We have a multiplicative structure
$$m:M(S)\ts M(T)\to M(S\xx T),\qquad [X\to S]\xx[Y\to T]\mto[X\xx Y\to S\xx T]$$
natural with respect to push-forward and pull-back.
It makes $K(\Var)=M(\pt)$ a commutative ring with the identity $\one=[\pt]$.
It also makes $K(\Var\qt S)=M(S)$ a module over $K(\Var)$.

Let us also define $K'(\Var\qt S)$ to be the localization of $K(\Var\qt S)$ with respect to $\bL=[\aa]\in K(\Var)$
and define $\KKB(\Var\qt S)$ to be the localization of $K(\Var\qt S)$ with respect to $\bL$ and $\bL^n-1$ for $n\ge1$
(where we consider $K(\Var\qt S)$ as a module over $K(\Var)$).
In particular, let $K'(\Var)=K'(\Var\qt\pt)$ and $\KKB(\Var)=\KKB(\Var\qt\pt)$.
We will call their elements motivic classes.
We will say that a motivic class is \idef{Tate} if it is a polynomial in $\bL^{\pm1}$ and we will say that it is \idef{quasi-Tate} if it is a polynomial in $\bL^{\pm1}$ divided by some powers of $\bL^n-1$ for $n\ge1$.

We define a ring homomorphism involution $K'(\Var)\to K'(\Var)$
by the rule
\begin{equation}
[X]\mto[X]\dual=\bL^{-\dim X}\cdot[X],
\end{equation}
for every smooth, projective and connected variety $X$ (note that such variety is automatically irreducible).
Note that $[\bP^1]=\one+\bL$ and $[\bP^1]\dual=\one+\bL\inv$, hence $\bL\dual=\bL^{-1}$.

\begin{remark}
\oper{Bl}
The above definition of the involution is motivated by taking duals in rigid monoidal categories of motives.
To see that the above involution is well-defined it is enough to show that it respects the blow-up relation \cite{bittner_universal}.
More precisely, let $Y\sbs X$ be a closed subvariety with both $X$ and $Y$ smooth, projective and connected.
Let $\hat X=\Bl_YX$ and $E$ be the exceptional divisor.
Then we have the blow-up relation $[\hat X]-[E]=[X]-[Y]$
and we need to show that
\begin{equation}\label{bl-rel-pres}
[\hat X]\dual-[E]\dual=[X]\dual-[Y]\dual.
\end{equation}
Let $n=\dim X$ and $k=\dim Y$.
Then $E$ is a projective bundle of dimension $n-k-1$
over $Y$ (see \eg \cite[B.6.3]{fulton_intersection}),
hence $(1-\bL)[E]=(1-\bL^{n-k})[Y]$.
Therefore
\begin{multline*}
\bL^n\cdot ([\hat X]\dual-[E]\dual)
=[\hat X]-\bL[E]=[X]-[Y]+(1-\bL)[E]\\
=[X]-\bL^{n-k}[Y]
=\bL^n([X]\dual-[Y]\dual).
\end{multline*}
This implies that \eqref{bl-rel-pres} is satisfied in $K'(\Var)$.
\end{remark}

Let us define the virtual Poincar\'e polynomial
$$P:K'(\Var)\to\bZ[t^{\pm\oh}],\qquad [X]\mto \sum_{p,q,n}(-1)^n h^{p,q}(H^n_c(X,\bC))t^{\oh(p+q)},$$
where $h^{p,q}(H^n_c(X,\bC))$ is the dimension of the $(p,q)$-type Hodge component of the mixed Hodge structure on $H^n_c(X,\bC)$.
In particular, $P(\bL;t)=t$ and if $X$ is smooth, projective and connected, then
$$P(X;t)=P([X];t)=\sum_n(-1)^n\dim H^n(X,\bC)t^{n/2}.$$
Note that in this case, by Poincar\'e duality, we have
$$P([X]\dual;t)=P([X];t\inv),$$
hence the same is true for any algebraic variety $X$.

\begin{remark}
\label{int coeff}
We can extend the virtual Poincar\'e polynomial to the algebra homomorphism $P:\KKB(\Var)\to\bQ(t^\oh)$, where $(\bL^n-1)\inv\mto(t^n-1)\inv$ for $n\ge1$.
Let $X$ be an algebraic variety such that $[X]=f(\bL)$ for some rational function $f\in\bQ(t)$.
Then $P(X;t)=f(t)$.
But $P(X;t)\in\bZ[t^\oh]$, hence $f(t)\in\bZ[t]$.
This implies that in order to show that the class of an algebraic variety is Tate, it is enough to show that it is quasi-Tate.
\end{remark}

Similarly, we define the $E$-polynomial
$$E:K'(\Var)\to\bZ[u^{\pm1},v^{\pm1}],\qquad
\sum_{p,q,n}(-1)^n h^{p,q}(H^n_c(X,\bC))u^pv^q.$$
In particular, $E(\bL;u,v)=uv$ and if $X$ is a smooth, projective and connected variety, then
$$E(X;u,v)=E([X];u,v)
=\sum_{p,q}h^{p,q}(H^{p+q}_c(X,\bC))(-u)^p(-v)^q.$$
Moreover, if $\dim X=d$, then $h^{p,q}=h^{d-p,d-q}$, hence
$$E([X]\dual;u,v)=E([X];u\inv,v\inv)$$
and the same is true for any algebraic variety $X$.

\medskip
More generally, let $\St$ be the $2$-category of algebraic stacks of finite type over $\bC$, having affine stabilizers \cite{bridgeland_introduction}.
For any algebraic stack $S$, locally of finite type over $\bC$, having affine stabilizers, one can define the Grothendieck group $K(\St\qt S)$ similarly to the above definition \cite{bridgeland_introduction}.
There is an isomorphism $K(\St)=K(\St\qt\pt)\iso \KKB(\Var)$
(see \eg \cite{bridgeland_introduction,joyce_motivic,toen_grothendieck}).

\medskip
Let $X$ be an algebraic variety and $\vi:E\to F$ be a morphism of vector bundles over $X$.
Define $Z(\vi)\sbs E$ to be the preimage $\vi\inv(0_F)$ of the zero section in $F$.

\begin{proposition}
Let $\vi:E\to F$ be a morphism of vector bundles and $\vi\dual:F\dual\to E\dual$ be the dual morphism.
Then
$$\bL^{\rk F}[Z(\vi)]=\bL^{\rk E}[Z(\vi\dual)].$$
\end{proposition}
\begin{proof}
Stratifying $X$, we can assume that $\vi$ has a constant rank $r\ge0$.
Then $[Z(\vi)]=[X]\cdot\bL^{\rk E-r}$ and 
$[Z(\vi\dual)]=[X]\cdot\bL^{\rk F-r}$.
\end{proof}

\begin{corollary}
\label{comparison2}
Given finite dimensional vector spaces $U,\,V$ and a morphism of algebraic varieties $\vi:X\to\Hom(U,V)$,
consider the induced morphisms of algebraic varieties
$$\vi_1:X\xx U\to V,\qquad \vi_2:X\xx V\dual\to U\dual$$
and let $Z(\vi_i)=\vi_i\inv(0)$.
Then
$$\bL^{\dim V}[Z(\vi_1)]=\bL^{\dim U}[Z(\vi_2)].$$
\end{corollary}


\sec[Exponential motivic theory]
Consider $\bA^1$ equipped with an abelian group structure $\mu:\aa\xx\aa\to\aa$, $(x,y)\mto x+y$.
Define a new motivic theory $M^e$ given by
$$M^e(S)=\Coker(M(S)\xto {p^*} M(S\xx\aa)),$$
where $p:S\xx\aa\to S$ is the projection.
We will write generators of $M^e(S)$ in the form 
$$[X\xto{(h,f)} S\xx\aa]=[X,f]_S.$$
The push-forward and pull-back morphisms extend automatically to $M^e$.
We define the multiplicative structure
\begin{gather*}
m:M^e(S)\ts M^e(T)\to M^e(S\xx T),\\
[X\to S\xx\aa]\xx[Y\to T\xx\aa]\mto[X\xx Y\to S\xx T\xx\aa\xx\aa\xto{1\xx 1\xx\mu}S\xx T\xx\aa].
\end{gather*}
The identity element $\one^e\in M^e(\pt)$ is given by $[\pt\xto0\bA^1]$.
We will call $M^e$ an exponential motivic theory.
One can identify $M^e(S)$ with the
Grothendieck group $K(\EVar\qt S)$ of varieties with exponentials over $S$ \cite{chambert-loir_motivic,wyss_motivic}.
There are injective morphisms
$$\iota
:M(S)\to M^e(S),\qquad [X]_S\mto[X,0]_S$$
which are functorial and preserve multiplicative structures.
The map $\iota$ is a section of the map $$\pi=(i_0^*-i_1^*):M^e(S)\to M(S),$$
where $i_t:S\to S\xx\aa,\, x\mto(x,t)$
is defined for every $t\in\bA^1$.
We will usually identify $M(S)$ with a subgroup of $M^e(S)$
using the above map $\iota$.

\begin{proposition}[\cf \cite{chambert-loir_motivic,wyss_motivic}]
\label{exp equation}
Let $X$ be an algebraic variety, $V$ be a finite-dimensional vector space, $\bar s:X\xx V\to\bA^1$ be a morphism linear on the second factor and $s:X\to V\dual$ be the induced morphism.
Then
$$[X\xx V,\bar s]_X=\bL^{\dim V}[Z(s)]_X,\qquad Z(s)=s\inv(0).$$
\end{proposition}

\begin{corollary}
Let $s\in\Ga(X,E\dual)$ be a section of a vector bundle, $Z(s)\sbs X$ be the zero locus and $\bar s:E\to\bA^1$ be the induced map.
Then we have an equality of exponential motivic classes
$$[E,\bar s]_X=\bL^{\rk E}[Z(s)]_X.$$
\end{corollary}

%
%
%

\sec[Jacobian algebras]
Given a quiver $Q$, we define a \idef{potential} to be an element $W\in \bC Q/[\bC Q,\bC Q]$.
It can be identified with a linear combination of cycles $W=\sum_u c_u u$,
where a cycle $u$ is a path $u=a_n\dots a_1$ such that $s(a_1)=t(a_n)$.
Here we consider cycles up to a cyclic shift, meaning that a cycle $a_n\dots a_1$ is equivalent to the cycle $a_i\dots a_1 a_n\dots a_{i+1}$ for all $1\le i<n$.
Given a cycle $u=a_n\dots a_1$ and an arrow $a\in Q_1$, we define the (cyclic) derivative
$$\frac{\dd u}{\dd a}=\sum_{i:a_i=a}a_{i-1}\dots a_1 a_n\dots a_{i+1}$$
and extend it to $\frac{\dd W}{\dd a}$ by linearity.
We define the \idef{Jacobian algebra} to be
\begin{equation}
J_{W}
=\bC Q/(\dd W)
=\bC Q/\rbr{\dd W/\dd a\rcol a\in Q_1}.
\end{equation}
Given a dimension vector $\bv\in\bN^{Q_0}$, consider the space of representations $\cR(Q,\bv)$.
For any representation $M\in\cR(Q,\bv)$ and for any cycle $u=a_n\dots a_1$, define 
$$\tr u|M=\tr(M_{a_n}\dots M_{a_1}).$$
Let us define the map (also called a potential)
$$\tr W:\cR(Q,\bv)\to \bC,\qquad M\mto \tr W|M=\sum_u c_u\cdot \tr u|M$$

\begin{proposition}(See \eg \cite{segal_a})
The critical locus of the map $\tr W:\cR(Q,\bv)\to \bC$
coincides with the space of representations $\cR(J_{W},\bv)\sbs\cR(Q,\bv)$.
\end{proposition}

We define a \idef{cut of the potential} $W=\sum_u c_u u$ to be a subset $I_1\sbs Q_1$ such that every cycle ~$u$ with $c_u\ne0$ has exactly one arrow from $I_1$ (appearing just once).
Let $\bar I_1=Q_1\ms I_1$ and let $I,\bar I\sbs Q$ be the corresponding subquivers with the sets of vertices $I_0=\bar I_0=Q_0$.
We define the \idef{partial Jacobian algebra}
\begin{equation}
J_{W,I}=\bC \bar I/(\dd W/\dd a\rcol a\in I).
\end{equation}
The map
$$\tr W:\cR(Q,\bv)=\cR(\bar I,\bv)\xx\cR(I,\bv)\to \bC$$
is linear on the second factor and induces a morphism of algebraic varieties
$$s:\cR(\bar I,\bv)\to \cR(I,\bv)\dual.$$
According to Proposition \ref{exp equation}, we have
$$[\cR(Q,\bv),\tr W]=\bL^{d_I(\bv)}\cdot [Z(s)],\qquad
d_I(\bv)=\dim\cR(I,\bv)=\sum_{(a:i\to j)\in I_1}\bbv_i\bbv_j.
$$

\def\pd#1#2{\frac{\dd#1}{\dd#2}}

\begin{proposition}
\label{zero locus}
The zero locus $Z(s)\sbs\cR(\bar I,\bv)$ coincides with the space of representations $\cR(J_{W,I},\bv)\sbs\cR(\bar I,\bv)$.
\end{proposition}
\begin{proof}
For any cycle $u$ in $W$, we have $u=\sum_{a\in I}a\pd ua$ (up to a cyclic shift) because of the assumption on $I$.
This implies that $\tr W|M=\sum_{a\in I}\tr a\pd{W}a |M$ for any representation $M\in\cR(Q,\bv)$.
Given $\bar M\in\cR(J_{W,I},\bv)\sbs\cR(\bar I,\bv)$, we have $\pd Wa|\bar M=0$ for all $a\in I$.
Therefore, for any $M\in\cR(I,\bv)$, the representation $N=(\bar M,M)\in\cR(Q,\bv)$ satisfies
$$\tr W|N=\sum_{a\in I}\tr \rbr{M_a\pd{W}a |\bar M}=0.$$
This implies $\bar M\in Z(s)$.
Conversely, if $\bar M\in Z(s)$, then for any $M\in\cR(I,\bv)$, representation $N=(\bar M,M)\in\cR(Q,\bv)$ satisfies $\tr W|M=0$.
Taking the derivatives with respect to all entries of $M_a$, we obtain $\pd Wa|\bar M=0$, for all $a\in I$.
This implies $\bar M\in\cR(J_{W,I},\bv)$.
\end{proof}

\begin{corollary}
\label{cut motive}
For any cut $I$ of the potential $W$, we have
$$[\cR(Q,\bv),\tr W]=\bL^{d_I(\bv)}\cdot [\cR(J_{W,I},\bv)].$$
\end{corollary}

This corollary implies that there is a relationship between the motivic classes $[\cR(J_{W,I},\bv)]$, for different cuts $I$.
We will need this relation only for disjoint cuts
and in this case we can give a proof that does not rely on exponential motives.

\begin{proposition}
\label{cut motive2}
Let $I_1,I'_1$ be two disjoint cuts of $W$.
Then
$$\bL^{d_I(\bv)}\cdot [\cR(J_{W,I},\bv)]
=\bL^{d_{I'}(\bv)}\cdot [\cR(J_{W,I'},\bv)]
.$$
\end{proposition}
\begin{proof}
Define $I''_1=Q_1\ms(I_1\cup I'_1)$.
Let 
$X=\cR(I'',\bv)$,
$U=\cR(I',\bv)$,
$V=\cR(I,\bv)$.
We have a map
$$\tr W:\cR=X\xx U\xx V\to \bC$$
which is linear on $U$ and $V$.
It induces maps
$$\vi_1:X\xx U\to V\dual,\qquad
\vi_2:X\xx V\to U\dual$$
and we conclude by Corollary \ref{comparison2} that
$$\bL^{\dim V}[Z(\vi_1)]=\bL^{\dim U}[Z(\vi_2)].$$
By Proposition \ref{zero locus} we have 
$\cR(J_{W,I},\bv)=Z(\vi_1)$, $\cR(J_{W,I'},\bv)=Z(\vi_2)$, hence the statement.
\end{proof}

\sec[Jacobian algebra of a translation quiver]
\label{jac of tq}
The following construction in the case of a double quiver $\bar Q$ can be found in \cite{ginzburg_calabi,mozgovoy_motivicb}.
Let $(Q,\ta)$ be a pair consisting of a quiver $Q$ with an automorphism $\ta:Q\to Q$.
Let $\Ga=Q^\ta$ be the corresponding twisted double quiver (see \S\ref{cuts}), which is a translation quiver with the cut $\Ga_1^+=Q_1\sbs\Ga_1$ (recall that for $a:i\to j$ in $Q_1$, we define $\si a=a^*:\ta j\to i$ and $\si (a^*)=\ta a:\ta i\to \ta j$).
We define a new quiver $\tl\Ga=\tl Q^\ta$ by adding to $\Ga=Q^\ta$ the arrows  
$$\ell_i:i\to\ta i,\qquad i\in \Ga_0.$$
For example, for any arrow $a:i\to j$ in $\Ga$, we have the following arrows in $\tl\Ga$
\begin{ctikzcd}
i\rar["a"]\dar["\ell_i"']&j\dar["\ell_j"]\\
\ta i\rar["\ta a"]&\ta j\ar[lu,"\si a"',close]
\end{ctikzcd}
We equip $\tl\Ga$ with the potential (the map $\eps:\Ga_1\to\bZ$ was defined in \S\ref{sec:mesh})
\begin{equation}
W=\sum_{(a:i\to j)\in\Ga_1}\eps(a) \ell_j\cdot  a\cdot \si a.
\end{equation}
Note that $W$ has the following three cuts 
\begin{equation}
\Ga^+_1,\qquad
\Ga_1^-=\si\Ga^+_1,\qquad
I_1=\sets{\ell_i}{i\in \Ga_0}
\end{equation}
such that $\tl\Ga_1=\Ga_1\sqcup I_1=\Ga_1^+\sqcup\Ga_1^-\sqcup I_1$.
We denote by $\Ga^+$, $\Ga^-$ and $I$ the corresponding subquivers of $\tl\Ga$ (which we will also call the cuts).

\begin{proposition}
\label{partial quotients}
Consider the cuts $\Ga^-$ and $I$ of the potential $W$ defined above. Then
\begin{enumerate}
\item $J_{W,I}\iso\Pi(\Ga)$, the mesh algebra.
\item The category of modules over $J_{W,\Ga^-}$ is equivalent to the category of pairs $(M,\vi)$, where $M$ is a representation of $\Ga^+$ and $\vi:M\to M^\ta$ is a homomorphism of representations.
\end{enumerate}
\end{proposition}
\begin{proof}
\clm1
For any $j\in\Ga_0$, we have $\pd W{\ell j}=\sum_{t(a)=j}\eps(a) a\si(a)$.
This is the mesh relation $e_j\fr e_{\ta j}$ which follows from the relation $\fr$.
Conversely, we have $\fr=\sum_j e_j\fr e_{\ta j}$.

\clm2
For any arrow $a:i\to j$ in $\Ga^+$, the arrow $\si a$ appears in the summands
$$\eps(a)\ell_j\cdot  a\cdot \si a,\qquad \eps(\si a)\ell_i \cdot\si a\cdot\ta a$$
of $W$.
Therefore
$$\pd W{(\si a)}=\eps(a)\rbr{\ell_j\cdot a-\ta a\cdot\ell_i}.$$
Given a module $M'$ over $J_{W,\Ga^-}$, we can restrict it to a representation $M$ of $\Ga^+$ (recall that $\tl \Ga_1=\Ga_1^+\sqcup\Ga_1^-\sqcup I_1$ and $J_{W,\Ga^-}$ is the quotient of the path algebra for the quiver with the set of arrows $\Ga_1^+\sqcup I_1$).
There are linear maps $\vi_i=M'_{\ell_i}:M_i\to M_{\ta i}=M^\ta_i$ for all $i\in\Ga_0$.
Because of the above relations, we have a commutative diagram
\begin{ctikzcd}
M_i\rar["M_a"]\dar["\vi_i"']&M_j\dar["\vi_j"]\\
M_{\ta i}\rar["M_{\ta a}"]&M_{\ta j}
\end{ctikzcd}
for all $a:i\to j$ in $\Ga^+$.
Therefore we obtain a homomorphism $\vi:M\to M^\ta$ of $\Ga^+$-representations. The converse construction is straightforward.
\end{proof}


Given a quiver $Q$ with an automorphism $\ta$, we define $\cR^\ta(Q,\bv)$ to be the space of pairs $(M,\vi)$, where $M\in\cR(Q,\bv)$ and $\vi\in\Hom(M,M^\ta)$, for any dimension vector $\bv\in\bN^{(Q_0)}$.
By the previous proposition, we can identify
$\cR^\ta(\Ga^+,\bv)$ with $\cR(J_{W,\Ga^-},\bv)$.

\begin{proposition}
\label{Pi to J}
We have
$$\bL^{\hi(\bv,\bv^\ta)}\cdot
[\cR(\Pi(\Ga),\bv)]=[\cR^\ta(\Ga^+,\bv)],$$
where $\hi$ is the Euler-Ringel form of the quiver $\Ga^+$.
\end{proposition}
\begin{proof}
By the previous proposition we have $$\cR(\Pi(\Ga),\bv)\iso\cR( J_{W,I},\bv),\qquad
\cR^\ta(\Ga^+,\bv)\iso \cR(J_{W,\Ga^-},\bv).$$
On the other hand by Proposition \ref{cut motive2} we have
$$\bL^{d_I(\bv)}\cdot [\cR(J_{W,I},\bv)]
=\bL^{d_{\Ga^-}(\bv)}\cdot [\cR(J_{W,\Ga^-},\bv)].
$$
Note that
$$d_I(\bv)-d_{\Ga^-}(\bv)
=\sum_{i\in\Ga_0}\bbv_i\bbv_{\ta i}-\sum_{(a:i\to j)\in\Ga_1^+}\bbv_{\ta j}\bbv_i
=\hi(\bv,\bv^\ta).
$$
\end{proof}

\subsection{Calculation of the motivic classes}

Our goal is to determine the motivic classes of
$\cR(\Pi(\Ga),\bv)$ and $\cR^\ta(\Ga^+,\bv)$
using the above relation, for certain translation quivers.

\begin{conjecture}
\label{conj tate rep space}
The motivic class of $\cR^\ta(Q,\bv)$ is Tate.
\end{conjecture}

\begin{remark}
If $\ta=\id$, then $\Ga=Q^\ta$ is the double quiver and $\Pi=\Pi(\Ga)$ is the preprojective algebra.
By Proposition \ref{Pi to J} we have $\bL^{\hi(\bv,\bv)}[\cR(\Pi,\bv)]=[\cR^\ta(Q,\bv)].$
These motivic classes were computed in \cite{mozgovoy_motivicb}, where they were expressed
as rational functions in $\bL$.
As $\cR^\ta(Q,\bv)$ are algebraic varieties, we conclude that the conjecture is true in this case (see Remark \ref{int coeff}).
\end{remark}

Let $Q$ be a quiver,
$\Ga=\bar Q$ be its double quiver and $\bd:\Ga_1\to\bZ$ be a weight map having the total weight $0\ne \om\in\bZ$.
Then we have $\tl\Ga=\loc_\bd(\bar Q)=\tl Q^{\ta_\om}$, where
$\tl Q=\loc_\bd(Q)$ and 
$$\ta_\om:\tl Q\to\tl Q,\qquad (i,n)\mto(i,n-\om),\qquad
a_n\mto a_{n-\om}.$$


\begin{theorem}
\label{calculation}
The motivic class of $\cR^{\ta_\om}(\tl Q,\tl\bv)$ is Tate, for any dimension vector $\tl\bv\in\bN^{(\tl Q_0)}$.
\end{theorem}
\begin{proof}
We will assume that $\om>0$.
An object in $\Rep^{\ta_\om}(\tl Q)$ is a pair $(M,\vi)$, where $M\in\Rep(\tl Q)$ and $\vi:M\to M^{\ta_\om}$ is a morphism.
This means that we have a collection of vector spaces $(M_{i,n})_{i\in Q_0,n\in\bZ}$ and compatible linear maps
$$M_{a,n}:M_{i,n-\bbd_a}\to M_{j,n},\qquad \vi_{i,n}:M_{i,n}\to M_{i,n-\om}$$
for $a:i\to j$ in $Q$, $i\in Q_0$ and $n\in\bZ$.
Let $A_\infty$ be the infinite quiver
$$\dots \to m+1\to m\to m-1\to\dots$$
For any representation $N$ of $A_\infty$ and for any $n\in\bZ$, we will write $N[n]$ for the shifted representation given by $N[n]_m=N_{m-n}$.
The maps $\vi_{i,n}$ induce representations $M^{i,k}$ of $A_\infty$
$$\dots\to M_{i,(m+1)\om+k}\to M_{i,m\om+k}\to M_{i,(m-1)\om+k}\to \dots$$
for all $i\in Q_0$ and $k\in\bZ$.
Note that $M^{i,k}[1]_m=M_{i,(m-1)\om+k}=M^{i,k-\om}_m$, hence $M^{i,k}[1]=M^{i,k-\om}$ and we require only representations $M^{i,k}$ for $0\le k<\om$.
The maps $M_{a,n}$ i
nduce homomorphisms of
$A_\infty$-representations $M^{a,k}:M^{i,k-\bbd_a}\to M^{j,k}$ for all $a:i\to j$ and $0\le k<\om$.

This implies that an object $M\in\Rep^{\ta_\om}(\tl Q)$ can be identified with a collection of $A_\infty$-representations $M^{i,k}$, where $i\in Q_0$ and $0\le k<\om$, and a collection of homomorphisms of $A_\infty$-representations
$$M^{a,k}:M^{i,k-\bbd_a}\to M^{j,k},\qquad a:i\to j,\, 0\le k<\om,$$
where we define $M^{i,m\om+k}=M^{i,k}[-m]$, for $m\in\bZ$ and $0\le k<\om$.

Let $S=\sets{(p,q)\in\bZ^2}{p\le q}$.
Indecomposable representations of $A_\infty$ are representations $I_{p,q}$, for $(p,q)\in S$, having dimension vector $\bv$ given by $\bbv_n=1$ for $p\le n\le q$ and $\bbv_n=0$ otherwise.
This implies that every representation of $A_\infty$ is isomorphic to a representation $I(\bm)=\bop_{s\in S}I_s^{\oplus \bm_s},$
where $\bm:S\to\bN$ is a map with finite support.
In our situation, we parametrize representations $M^{i,k}$ by a map $\bm:Q_0\xx\bZ_\om\xx S\to\bN$ with finite support, so that
$$M^{i,k}=I(\bm_{i,k})=\bop_{s\in S}I_{s}^{\oplus \bm_{i,k,s}}.$$
The corresponding dimension vector $\tl\bv=\udim M\in\bN^{(\tl Q_0)}$ is given by
$$\tl\bv_{i,m\om+k}=\nn\bm_{i,m\om+k}=\sum_{p\le m\le q}\bm_{i,k,p,q},\qquad i\in Q_0,\, m\in\bZ,\, 0\le k<\om.$$

Consider the motivic class
$$\fc(\bm)=\frac{\prod_{a\rcol i\to j,0\le k<\om}[\Hom(M^{i,k-\bbd_a},M^{j,k})]}{\prod_{i\in Q_0,0\le k<\om}[\Aut(M^{i,k})]}$$
which is a polynomial in $\bL^{\pm1}$ divided by powers of $(\bL^n-1)$, $n\ge1$ (see Proposition \ref{aut motive}).
We obtain
$$
\frac{[\cR^{\ta_\om}(\tl Q,\tl\bv)]}{[\GL_{\tl\bv}]}
=\sum_{\ov{\bm :Q_0\xx\bZ_\om\xx S\to\bN}{\nn \bm=\tl\bv}}\fc(\bm).
$$
This implies that the motivic class of $\cR^{\ta_\om}(\tl Q,\tl\bv)$ is a polynomial in 
$\bL^{\pm1}$ divided by powers of $(\bL^n-1)$, $n\ge1$.
But as $\cR^{\ta_\om}(\tl Q,\tl\bv)$ is an algebraic variety, we conclude that its motivic class is a polynomial in $\bL$ with integer coefficients (see Remark \ref{int coeff}).
\end{proof}

Note that the above proof gives an explicit formula for  $[\cR^{\ta_\om}(\tl Q,\tl\bv)]$.

\begin{theorem}
\label{class of stack2}
Let $\bd:\bar Q_1\to\bZ$ be a weight map having the total weight $\om\ne0$, $\tl\Ga=\loc_\bd(\bar Q)$ and $\tl\bv\in\bN^{(\Ga_0)}$.
Then the motivic class of $\cR(\Pi(\tl\Ga),\tl\bv)$ is Tate.
\end{theorem}
\begin{proof}
We have $\tl\Ga=\tl Q^{\ta_\om}$, where
$\tl Q=\loc_\bd(Q)$.
By Proposition \ref{Pi to J} we have
$$\bL^{\hi(\tl\bv,\tl\bv^{\ta_\om})}\cdot  [\cR(\Pi(\tl\Ga),\tl\bv)]=
[\cR^{\ta_\om}(\tl Q,\tl\bv)],$$
where $\hi$ is the Euler-Ringel form of $\tl Q$.
Now we apply the previous theorem.
\end{proof}

\begin{remark}
The previous result implies that 
$\cR(\Pi(\bZ Q),\bv)$ has a Tate motivic class, where $\Ga=\bZ Q$ is the repetition quiver and $\bv\in\bN^{(\Ga_0)}$.
Indeed, we have seen that $\Ga=\loc_\bd(\bar Q)$, where $\bd:\bar Q_1\to\bZ$ is defined by $\bbd_a=0$ and $\bbd_{\si a}=1$, for $a\in Q_1$.
\end{remark}

\begin{proposition}[See \cite{mozgovoy_motivicb}]
\label{aut motive}
Let $M$ be a finite dimensional module over an algebra and let $M=\bop_{i\in I}M_i^{\oplus n_i}$ be its decomposition into indecomposable summands, with $M_i\not\iso M_j$ for $i\ne j$.
Then
$$\frac{[\Aut M]}{[\End(M)]}=\prod_{i\in I}(\bL\inv)_{n_i},$$
where $(q)_n=(q;q)_n=\prod_{k=1}^n(1-q^k)$ is the $q$-Pochhammer symbol.
\end{proposition}

\sec[Wall-crossing formula]
Let $\Ga$ be a translation quiver with a cut $\Ga^+$ and $\Pi=\Pi(\Ga)$ be the mesh algebra.
Let $\te\in\bR^{(\Ga_0)}$ be a stability parameter and $\bv\in\bN^{(\Ga_0)}$ be a dimension vector.
Recall that $\cR_\te(\Pi,\bv)\sbs\cR(\Pi,\bv)$ is an open subspace of $\te$-semistable representations.
Consider the corresponding algebraic stacks
$$\fM_\te(\bv)=\cR_\te(\Pi,\bv)/\GL_\bv,\qquad
\fM(\bv)=\cR(\Pi,\bv)/\GL_\bv.$$
Note that usually, given an algebraic group $G$ acting on an algebraic variety $X$, one denotes the corresponding quotient stack by $[X/G]$.
We depart from this convention and write $X/G$ instead as we use square brackets for the motivic classes.
Our goal is to prove a relation between the motivic classes of the above stacks.


\begin{theorem}
Assume that $\te^\ta=\te$. Then
$$[\fM(\bv)]=\sum_{\ov{\tp\bv 1+\dots+\tp\bv n=\bv}{\mu_\te(\tp\bv 1)>\dots>\mu_\te(\tp\bv n)}}\bL^{\sum_{i<j}\nu(\tp\bv i,\tp\bv j)}\prod_i[\fM_\te(\tp\bv i)],$$
where $\nu(\bu,\bv)=-\hi(\bv,\bu)-\hi(\bu,\bv^\ta)$ and $\hi$ is the Euler-Ringel form $\Ga^+$.
\end{theorem}
\begin{proof}
Let $\fM=\bigsqcup_\bv\fM(\bv)$ be the stack of all representations of $\Pi$.
Let $H=K(\St\qt\fM)$ be the motivic Hall algebra of the category $\mmod\Pi(\Ga)$ (see \eg \cite{bridgeland_hall})
and let $\hat H$ be its completion with respect to the filtration arising from the grading by $\bN^{(\Ga_0)}$.
Then we have the following relation in $\hat H$
(see \eg \cite{joyce_configurations2,reineke_harder-narasimhan})
$$[\fM\to\fM]=\sum_{\mu_\te(\tp\bv 1)>\dots>\mu_\te(\tp\bv n)}
[\fM_\te(\tp\bv 1)\to\fM]*\dots *[\fM_\te(\tp\bv n)\to\fM].
$$
Next we define a non-commutative algebra
$$A=\bop_{\bv\in\bN^{(\Ga_0)}} \cV\cdot t^\bv,\qquad \cV=K(\St)\iso\KKB(\Var),$$
with the multiplication
$$t^\bu*t^\bv=\bL^{\nu(\bu,\bv)}t^{\bu+\bv},$$
where $\nu$ is the bilinear form defined above.
Let $\hat A$ be the completion of $A$.
We define the integration map
$$I:\hat H\to\hat A,\qquad [\cX\to\fM(\bv)]\mto [\cX]\cdot t^\bv.$$
This map is not necessary an algebra homomorphism, but it satisfies
$$I([\cX\to\fM]*[\cY\to\fM])
=I([\cX\to\fM])*I([\cY\to\fM])$$
under the assumption that
$$h^0(N,M)-h^1(N,M)=-\nu(\udim M,\udim N)$$
for all $M\in\cX$ and $N\in\cY$ (see \eg \cite{joyce_configurations2,reineke_harder-narasimhan}),
where we identify $M,N$ with their images in ~$\fM$.

In our situation, we consider $M\in\fM_{\ge c}$ 
(the stack of objects with all Harder-Narasimhan factors having slope $\ge c$) and $N\in\fM_{<c}$
(the stack of objects with all Harder-Narasimhan factors having slope $<c$).
As $\te^\ta=\te$, the module $N^\ta$ has the same slope as $N$ and it is also contained in $\fM_{<c}$.
Therefore $\Hom(M,N^\ta)=0$ and we obtain from Proposition \ref{h^i} that
$$h^0(N,M)-h^1(N,M)=\hi(N,M)+\hi(M,N^\ta)
=-\nu(\udim M,\udim N).$$
Applying the integration map we obtain
$$\sum_{\bv}[\fM(\bv)]t^\bv=\sum_{\mu_\te(\tp\bv 1)>\dots>\mu_\te(\tp\bv n)}
\bL^{\sum_{i<j}\nu(\tp\bv i,\tp\bv j)}\cdot 
\prod_i[\fM_\te(\tp\bv i)]\cdot t^{\sum_i\tp\bv i}.
$$
This formula is equivalent to the statement of the theorem.
\end{proof}

The above recursion formula can be solved (see \eg \cite{reineke_harder-narasimhan}) and we can express motivic classes of $\fM_\te(\bv)$ in terms of motivic classes of $\fM(\bv)$.

\begin{theorem}
\label{wall-cross2}
Assume that $\te^\ta=\te$. Then
$$[\fM_\te(\bv)]=\sum_{\ov{\tp\bv 1+\dots+\tp\bv n=\bv}{\mu_\te(\tp\bv 1+\dots+\tp\bv k)>\mu_\te(\bv)\,\forall k<n}}\bL^{\sum_{i<j}\nu(\tp\bv i,\tp\bv j)}\prod_i[\fM(\tp\bv i)],$$
\end{theorem}

\begin{corollary}
\label{quasi-tate stack}
Let $Q$ be a quiver, $\bd:\bar Q_1\to\bZ$ be a weight map having the total weight $\om\ne0$ and $\tl\Ga=\loc_\bd(\bar Q)$ be the corresponding localization quiver.
For any stability parameter $\tl\te\in\bR^{\tl\Ga_0}$ with $\tl\te^\ta=\tl\te$ and for any dimension vector $\tl\bv$,
the motivic class $[\fM_{\tl\te}(\Pi(\tl\Ga),\tl\bv)]$ is quasi-Tate.
\end{corollary}
\begin{proof}
By Theorem \ref{class of stack2} the stack $\fM(\Pi(\tl\Ga),\tl\bv)
=\cR(\Pi(\tl\Ga),\tl\bv)/\GL_{\tl\bv}$ has a quasi-Tate motivic class.
Now we apply the above theorem.
\end{proof}

\begin{remark}
Because of Conjecture \ref{conj tate rep space} we expect that if $\te^\ta=\te$, then $\fM(\Pi(\Ga),\bv)$ and $\fM_\te(\Pi(\Ga),\bv)$ have quasi-Tate motivic classes for an arbitrary translation quiver \Ga with a cut.
We will see later that $[\fM_\te(\Pi(\Ga),\bv)]$ is quasi-Tate at least if $\te$ is \bv-generic, $\te^\ta=\te$ and $\bv^\ta\ne\bv$.
\end{remark}

%% file: torus_action.tex
\section{Torus action}
In this section we will consider only algebraic varieties over $\bC$.

\sec[\BB decomposition]
\label{birula}
Let $X$ be an algebraic variety equipped with an action of $T=\Gm$.
Define the $\pm$-attractors
$$X^+=\sets{x\in X}{\exists\,\lim_{t\to0}tx},\qquad
X^-=\sets{x\in X}{\exists\,\lim_{t\to\infty}tx}.$$
Assume that $X$ is smooth and can be covered by $T$-invariant quasi-affine open subvarieties.
For any point $x\in X^T$, the tangent space $T_xX$ is equipped with a $T$-action and there is a weight space decomposition (for positive, zero and negative weights)
$$T_xX=T_x^+X\oplus T_x^0X\oplus T_x^-X.$$
The fixed locus $X^T$ is smooth and $T_xX^T=T_x^0X$.
In what follows, we define an affine bundle over an algebraic variety $Y$ to be a fiber bundle $F\to Y$,
locally trivial in the Zariski topology,
with affine spaces as fibers (note that $F\to Y$ is not necessarily a vector bundle).
We define $\rk F$ to be the dimension of the fibers (it is constant if $Y$ is connected) or, equivalently, the relative dimension of the morphism $F\to Y$.

\begin{theorem}[See \cite{bialynicki-birula_some,hesselink_concentration}]
\label{BB}
For any connected component $X_i\sbs X^T$, the spaces
$$X_i^+=\sets{x\in X}{\exists\,\lim_{t\to0}tx\in X_i},\qquad
X_i^-=\sets{x\in X}{\exists\,\lim_{t\to\infty}tx\in X_i}$$
are locally closed and are affine bundles over $X_i$.
For all $x\in X_i$, we have
$$T_x X_i^+=T^+_xX\oplus T^0_xX,\qquad T_x X_i^-=T^-_xX\oplus T^0_xX$$
and $\rk X_i^+=\dim T^+_xX$, $\rk X^-_i=\dim T^-_xX$.
There exists an ordering $X_1,\dots,X_n$ of the connected components of $X^T$ such that
$$X_{\le k}^+=\bigsqcup_{i\le k}X_i^+,
\qquad
X_{\ge k}^-=\bigsqcup_{i\ge k}X_i^-$$
are closed in $X$, for all $1\le k\le n$.
We have decompositions $X^+=\bigsqcup_i X_i^+$,
$X^-=\bigsqcup_i X_i^-$.
\end{theorem}

\begin{proposition}
\label{dual1}
Assume that $X$ is smooth and $X^T$ is projective. Then
$$[X^-]\dual=\bL^{-\dim X}\cdot[X^+].$$
In particular, $X^+$ has a Tate motivic class if and only if $X^-$ does.
\end{proposition}
\begin{proof}
For every connected component $X_i\sbs X^T$, let 
$d_i=\dim X_i$, $d_i^+=\rk X_i^+$, $d_i^-=\rk X_i^-$,
so that $d_i+d_i^+ +d_i^-=\dim X$.
Then
$$[X_i^-]\dual
=\bL^{-d^-_i}[X_i]\dual
=\bL^{-d_i-d^-_i}[X_i]
=\bL^{-\dim X}[X_i^+].
$$
Taking the sum over all connected components, we obtain the required result.
\end{proof}

\begin{corollary}
\label{virtual1}
Assume that $X$ is smooth and $X^T$ is projective.
Then the virtual Poincar\'e polynomials of attractors satisfy
$$P(X^-;t\inv)=t^{-\dim X}P(X^+;t).$$
\end{corollary}

We will say that an algebraic variety $X$ is \idef{pure} if the mixed Hodge structure on $H^i_c(X,\bC)$ is pure of weight $i$ for all $i$.

\begin{proposition}[\cf \cite{crawley-boevey_absolutely}]
\label{pure}
Assume that $X$ is smooth and $X^T$ is projective.
Then $X^+$ and $X^-$ are pure.
\end{proposition}
\begin{proof}
Connected components $X_k\sbs X^T$ are smooth and projective, hence pure.
Therefore $X_k^+$ and $X_k^-$, which are affine bundles over $X_k$, are also pure.
Let us order the connected components $X_1,\dots,X_n$ as in Theorem \ref{BB}.
We will show that $X^+_{\le k}=\bigsqcup_{j\le k}X_j^+$ are pure by induction. Then $X^+=X^+_{\le n}$ is pure. The proof for $X^-$ is similar.
The subvariety $X^+_{<k}\sbs X^+_{\le k}$ is closed and has the open complement $X^+_k$.
We obtain a long exact sequence
$$
H_c^{i-1}(X^+_{<k},\bC)\xto\dd
H_c^{i}(X^+_{k},\bC)\to
H_c^{i}(X^+_{\le k},\bC)\to
H_c^{i}(X^+_{<k},\bC)\xto\dd
H_c^{i+1}(X^+_{k},\bC)\to
$$
By purity of $X^+_k$ and $X^+_{<k}$, the connection maps $\dd$ are zero, hence we have an exact sequence
$$0\to 
H_c^{i}(X^+_{k},\bC)\to
H_c^{i}(X^+_{\le k},\bC)\to
H_c^{i}(X^+_{<k},\bC)\to
0$$
and $H_c^{i}(X^+_{\le k},\bC)$ is pure of weight $i$.
\end{proof}



\subsection{Fixed point varieties}
\label{fixed locus}
Let $\Ga$ be a translation quiver with a cut
and $\Pi=\Pi(\Ga)$ be the corresponding mesh algebra.
Let $\bv\in\ds\bN{\Ga_0}$ be a dimension vector and $\te\in\bR^{\Ga_0}$ be a 
stability parameter.
Our goal is to study torus actions on $\cM=\cM^\s_\te(\Pi,\bv)$  
and to interpret the fixed locus
as a disjoint union of translation quiver varieties
(for a new translation quiver).

Let $\La=\bZ^r$ be a free abelian group of finite rank and $T=\Hom_\bZ(\La,\bC^*)\iso\Gm^r$ be the corresponding torus with the character group $X^*(T)=\Hom(T,\Gm)\iso\La$.
For any $t\in T$ and $n\in\La$, we define $t^n=t(n)\in \bC^*$.
Let $\bd:\Ga_1\to\La$ be a weight map having the total weight $\om\in\La$ ~\S\ref{localiz}.
We define the action of $T$ on $\cM$ by the rule
$$t\cdot M=(t^{\bbd_a}M_a)_{a\in\Ga_1},\qquad t\in T,\, M\in\cM.$$
To be more precise, the above definition gives an action of $T$ on $\cR(\Ga,\bv)$.
If $M\in \cR(\Pi,\bv)\sbs\cR(\Ga,\bv)$
and $M'=t\cdot M$, then $M'_a M'_{\si a}=t^\om M_aM_{\si a}$, hence the representation $M'$ satisfies the mesh relation and $M'\in \cR(\Pi,\bv)$.
This implies that we have induced actions of $T$ on the subsets
$$\cR_\te^\st(\Pi,\bv)\sbs \cR_\te(\Pi,\bv)
\sbs \cR(\Pi,\bv)\sbs\cR(\Ga,\bv).$$
These actions descend to the actions of $T$ on (see \S\ref{sec:QV})
$$
\cM=\cM_\te^\st(\Pi,\bv)=\cR_\te^\st(\Pi,\bv)\qt\GL_\bv,
\qquad
\cM_\te(\Pi,\bv)=\cR_\te(\Pi,\bv)\GIT\GL_\bv.
$$

We will show that $T$-fixed points of $\cM$ can be interpreted as representations of the localization quiver $\tl\Ga=\loc_\bd(\Ga)$ defined in \S\ref{localiz}.
Recall that $\tl\Ga_0=\Ga_0\xx\La$ and arrows of $\tl\Ga$ are of the form
$$a_n:(i,n-\bbd_a)\to(j,n),\qquad (a:i\to j)\in\Ga_1,\, n\in\La.$$
Note that there is an action of $\La$ on $\tl\Ga$ which induces an isomorphism $\tl\Ga/\La\iso\Ga$.
It also induces an action of $\La$ on $\bN^{(\tl\Ga_0)}$.
We have a morphism of translation quivers (\cf \S\ref{coverings})
$$\pi:\tl\Ga\to\Ga,\qquad (i,n)\mto i,\qquad a_n\mto a,$$
and an exact functor 
\begin{gather}
\label{fun pi*}
\pi_*:\Rep\tl\Ga\to\Rep\Ga,\qquad \tl M\mto M,\\ M_i=\bop_{n\in\La}\tl M_{i,n},\qquad M_a=\sum_{n\in\La}\tl M_{a,n},\qquad i\in \Ga_0,\,a\in\Ga_1,
\end{gather}
which induces an exact functor $\pi_*:\mmod\Pi(\tl\Ga)\to\mmod\Pi(\Ga)$.
Let us also consider the maps
\begin{gather*}
\pi_*:\bN^{(\tl\Ga_0)}\to\ds\bN{\Ga_0},\qquad
\tl\bv\mto\bv,\qquad
\bbv_i=\sum_{k\in\pi\inv(i)}\tl \bbv_{k},\\
\pi^*:\bR^{\Ga_0}\to\bR^{\tl\Ga_0},\qquad
\te\mto\tl\te,\qquad
\tl\te_k=\te_{\pi(k)}.
\end{gather*}
Note that if $\bv=\pi_*(\tl\bv)$ and $\tl\te=\pi^*(\te)$, then 
$$\te\cdot\bv
=\sum_{i\in \Ga_0}\te_i\cdot\sum_{k\in\pi\inv (i)}\tl \bbv_k
=\sum_{k\in\tl\Ga_0}\te_{\pi(k)}\tl \bbv_k=\tl\te\cdot\tl\bv
$$
and $\mu_\te(\bv)=\mu_{\tl\te}(\tl\bv)$.

An analogue of the following statement in the case of Nakajima quiver varieties can be found in ~\cite{nakajima_quiverb,nakajima_quivere}
and in the case of quiver varieties without relations in \cite{reineke_localization,weist_localization}.

\begin{theorem}
\label{decomp}
Let $\bv\in\ds\bN{\Ga_0}$ be a dimension vector,
$\te\in\bR^{\Ga_0}$ and $\tl\te=\pi^*(\te)$.
Then there is a morphism 
$$\bar\io:\bigsqcup_{\tl\bv\in\pi_*\inv(\bv)/\La}
\cM^\s_{\tl\te}(\Pi(\tl\Ga),\tl\bv)\to
\cM_\te^\s(\Pi(\Ga),\bv)^T.$$
of algebraic schemes such that the image of every
$\cM^\s_{\tl\te}(\Pi(\tl\Ga),\tl\bv)$ is a union of connected components of $\cM_\te^\s(\Pi(\Ga),\bv)^T$
and $\bar\io$ induces a bijection between the sets of \bC-valued points.
If all of the above moduli spaces are smooth and equidimensional, then $\bar\io$ is an isomorphism.
\end{theorem}


\begin{proof}
Let $\cM=\cM^\s_\te(\Pi,\bv)$ 
and let $M\in \cM^T$ be a fixed point (more precisely, its representative in $\cR^\st_\te(\Pi,\bv)$).
Note that $M$ is $\te$-stable, hence $\End(M)=\bC$.
For every $t\in T$, the representation $t\cdot M$ is isomorphic to~$M$, hence there exists a unique element $\bar \rho(t)\in G_\bv=\GL_\bv/\Gm$ such that
\begin{equation}\label{fixed point}
t\cdot M=\bar \rho(t)\cdot M.
\end{equation}
We obtain a group homomorphism $\bar\rho:T\to G_\bv$.
Let $\cM(\bar\rho)\sbs\cM^T$ be the set of points satisfying ~\eqref{fixed point} for some representative in $\cR^\st_\te(\Pi,\bv)$.
The set $\cM(\bar\rho)$ depends only on the $G_\bv$-conjugacy class of ~$\bar\rho$.
The $G_\bv$-conjugacy class of $\bar\rho$ is locally constant on $\cM^T$ by Lemma ~\ref{lm:conj class}.
Therefore $\cM(\bar\rho)$ is a union of connected components of $\cM^T$. 
In particular, $\cM(\bar\rho)$ is closed in $\cM$.

The group homomorphism $\bar\rho:T\to G_\bv$
can be lifted to a group homomorphism $\rho:T\to\GL_\bv$ (see \eg \cite[Lemma 3.3]{weist_localization}).
Given another lift $\rho':T\to\GL_\bv$,
we obtain $\rho'\rho\inv:T\to\Gm$ which corresponds to an element of $\La$.
Using the lift $\rho:T\to\GL_\bv$, we can consider every vector space ~$M_i$ as a representation of $T$, hence as a \La-graded vector space by taking the corresponding weight space decomposition
$$M_i=\bop_{n\in\La}M_{i,n},\qquad i\in \Ga_0,$$
so that $\rho(t)_i x=t^n x$ for $x\in M_{i,n}$ and $t\in T$.
These \La-gradings are uniquely determined by the $\GL_\bv$-conjugacy class of $\rho$.
The dimension vector
$$\tl\bv:\tl\Ga_0=\Ga_0\xx\La\to\bN,\qquad (i,n)\mto \dim M_{i,n},$$
satisfies $\pi_*(\tl\bv)=\bv$.
It is determined only up to a translation by $\La$ because of the non-uniqueness of the lift $\rho:T\to\GL_\bv$.
We obtain a $1-1$ correspondence between the set of $G_\bv$-conjugacy classes of $\bar\rho:T\to G_\bv$ and the set of dimension vectors $\tl\bv\in\pi_*\inv(\bv)$ up to the action of \La.

For every arrow $a:i\to j$ in $\Ga$, the linear map $M_a:M_i\to M_j$ satisfies
$$t^{\bbd_a}M_a=(t\cdot M)_a=(\rho(t)\cdot M)_a=\rho(t)_j M_a \rho(t)_i\inv.$$
This implies that its component $M_{a,m,n}:M_{i,m}\to M_{j,n}$ satisfies
$t^{\bbd_a}M_{a,m,n}=t^{n-m}M_{a,m,n}$ and we can have nonzero components only for $m=n-\bbd_a$.
We define 
$$M_{a,n}=M_{a,n-\bbd_a,n}:M_{i,n-\bbd_a}\to M_{j,n},\qquad
(a:i\to j)\in\Ga_1,\, n\in\La.$$
In this way we obtain a representation 
$\tl M$
of the localization quiver $\tl\Ga$ such that $\pi_*(\tl M)=M$.
The translation quiver $\tl\Ga$ has a cut (see \S\ref{localiz}) and the representation $\tl M$ satisfies the corresponding mesh relation.
Moreover, it is stable with respect to the stability parameter $\tl\te=\pi^*(\te)\in\bR^{\tl\Ga_0}$
as every subrepresentation of $\tl M$ induces a subrepresentation of $M$ (see ~\eqref{fun pi*}).
We conclude that 
$\tl M\in\cM^\s_{\tl\te}(\Pi(\tl\Ga),\tl\bv)$.

Let us show that if $\tl M$ is $\tl\te$-stable, then $M=\pi_*(\tl M)$ is $\te$-stable (\cf \cite{reineke_localization}).
By construction, we have an isomorphism $\rho(t):M\to tM$ for all $t\in T$
\begin{ctikzcd}
M_i\rar["M_a"]\dar["\rho(t)_i"']&M_j\dar["\rho(t)_j"]\\
M_i\rar["(tM)_a"]&M_j
\end{ctikzcd}
If $M$ is not $\te$-stable, let $N\subsetneq M$ be some stable submodule with $\mu_\te(N)\ge\mu_\te(M)$.
We have a stable submodule $tN\sbs tM$ (obtained by changing the maps $N_a$, but leaving the spaces $N_i\sbs M_i$ unchanged)
and an isomorphic submodule $N_t=\rho(t)\inv (tN)\sbs M$ (obtained by changing the spaces $N_i$, and restricting the maps $M_a$ to them).
The submodules $N_t$, for $t\in T$, form a direct sum in $M$, hence there are finitely many of them.
But they form a continuous $T$-family in the product of Grassmannians of $M_i$.
Therefore we have $N_t=N$ for all $t\in T$.
This implies that $\rho(t)$ restricts to an isomorphism $\rho(t):N\to tN$ and
repeating the above construction we obtain that $N=\pi_*(\tl N)$ for some submodule $\tl N\sbs\tl M$.
We have $\mu_{\tl\te}(\tl N)=\mu_\te(N)\ge \mu_\te(M)=\mu_{\tl\te}(\tl M)$, hence $\tl M$ is not stable, a contradiction.

The functor $\pi_*$ defined in \eqref{fun pi*} induces a closed embedding
$\io_{\tl\bv}:\cR(\Pi(\tl\Ga),\tl\bv)\emb
\cR(\Pi(\Ga),\bv)$.
By the above argument, it restricts to a closed embedding
$\io_{\tl\bv}:\cR_{\tl\te}^\st(\Pi(\tl\Ga),\tl\bv)\emb
\cR_\te^\st(\Pi(\Ga),\bv)$
which descends to a morphism between moduli spaces
\begin{ctikzcd}
\cR_{\tl\te}^\st(\Pi(\tl\Ga),\tl\bv)
\rar["\io_{\tl\bv}"]\dar[]&
\cR_\te^\st(\Pi(\Ga),\bv)\dar[]\\
\cM_{\tl\te}^\st(\Pi(\tl\Ga),\tl\bv)
\rar["\bar\io_{\tl\bv}"]&
\cM=\cM_\te^\st(\Pi(\Ga),\bv)
\end{ctikzcd}
By the above discussion, 
we obtain a morphism
$\bar\io_{\tl\bv}:\cM^\s_{\tl\te}(\Pi(\tl\Ga),\tl\bv)
\to\cM(\bar\rho)\sbs\cM^T$
which induces a bijection between the sets of $\bC$-valued points.
If $\cM$ is smooth, then $\cM(\bar\rho)$ is also smooth.
If, moreover, $\cM^\s_{\tl\te}(\Pi(\tl\Ga),\tl\bv)$ is smooth and equidimensional, then
$\bar\io_{\tl\bv}:\cM^\s_{\tl\te}(\Pi(\tl\Ga),\tl\bv)
\to\cM(\bar\rho)$ is an isomorphism.
\end{proof}

\begin{lemma}\label{lm:conj class}
Let $T$ be an algebraic torus, $G$ be a reductive group and let $T\xx G$ act on an affine variety ~$X$ so that $G$ acts freely on $X$ and $T$ acts trivially on the geometric quotient $X/G$.
For any $x\in X$, 
let $\rho_x:T\to G$ be the unique group homomorphism such that $tx=\rho_x(t)x$ for all $t\in T$.
Then the $G$-conjugacy class of $\rho_x$ is locally constant on $X$.
\end{lemma}
\begin{proof}
Let $G$ be a reductive group acting on an affine variety $X$ so that all orbits are closed.
By \cite[Proposition 5.5]{drezet_lunas},  
for any $x\in X$ there exists a $G$-invariant open subset $x\in U\sbs X$ such that for any $y\in U$, the group $G_y$ is the conjugate of a subgroup of $G_x$.
Let us apply this statement to the group $\bar G=T\xx G$ acting on $X$.
Then, for any $y\in U$, there exists $g\in G$ such that $g\bar G_yg\inv \sbs \bar G_x$.
We have $\bar G_x=\sets{(t\inv,\rho_x(t))}{t\in T}\iso T$ and similarly $\bar G_y\iso T$.
This implies that $g\bar G_y g\inv=\bar G_x$, hence
$g\rho_y(t)g\inv=\rho_x(t)$ for all $t\in T$.
\end{proof}

\begin{remark}
A morphism $f:X\to Y$ between algebraic schemes over $\bC$ is called a \idef{geometric bijection} if it induces a bijection between the sets of $\bC$-valued points.
It follows from \cite[Lemma 2.8]{bridgeland_introduction} that, for any geometric bijection $f:X\to Y$, 
we have an equality $[X]=[Y]$ in the Grothendieck group of algebraic varieties.
In particular, by the previous theorem, we have
$[\cM_\te^\s(\Pi(\Ga),\bv)^T]
=\sum_{\tl\bv\in\pi_*\inv(\bv)/\La}
[\cM^\s_{\tl\te}(\Pi(\tl\Ga),\tl\bv)]$.
\end{remark}

\subsection{Character of the tangent space}
Let us use the same notation as in the previous section.
We assume that $\te$ is $\bv$-generic, $\te^\ta=\te$ and $\bv^\ta\ne\bv$, so that the moduli space $\cM=\cM_\te(\Pi,\bv)=\cM_\te^\st(\Pi,\bv)$ is smooth by Theorem
\ref{smoothness2}.
Let $\tl M\in\cM^\st_{\tl\te}(\Pi(\tl\Ga),\tl\bv)$ and $M=\pi_*(\tl M)\in \cM^T$.
Then the tangent space $T_M\cM\iso\Ext^1_\Pi(M,M)$ is equipped with the $T$-action.
We will determine its class in the Grothendieck group $K_0(\Rep T)$.

The category of $T$-representations can be identified with the category of finite-dimensional $\La$-graded vector spaces, where $\La=X^*(T)$ is the group of characters.
Given a $T$-representation ~$V$, we consider its
weight space decomposition $V=\bop_{n\in\La}V_n$
as a \La-grading.
The Grothendieck group $K_0(\Rep T)$ is
equipped with the commutative ring structure given by $[V]\cdot [W]=[V\ts W]$ and the involution given by
$[V]^*=[DV]$, $DV=\Hom_\bC(V,\bC)$.
The character map
\begin{equation}
\chr:K_0(\Rep T)\isoto\bZ[\La]=\bop_{n\in\La}\bZ e^n,\qquad [V]\mto \sum_{n\in\La}\dim V_n\cdot e^n,
\end{equation}
is an isomorphism of rings
and the involution on $K_0(\Rep T)$ corresponds to the involution on $\bZ[\La]$ given by $(e^n)^*=e^{-n}$.

Let $\Vect^\La$ be the category of \La-graded vector spaces (with morphisms of degree $0$).
Given $V\in\Vect^\La$ and $n\in\La$, we define the shifted object
\begin{equation}
V[n]\in\Vect^\La,\qquad
V[n]_m=V_{m-n}\quad \forall m\in\La.
\end{equation}
Given $V,W\in\Vect^\La$, we define the internal \Hom-object
\begin{gather}
\lHom(V,W)=\bop_{n\in\La}\Hom_n(V,W)\in\Vect^\La,\\
\Hom_n(V,W)=\prod_{m\in\La}\Hom(V_m,W_{m+n})=\Hom(V,W[-n]).
\end{gather}
In particular,
we define the dual $DV=\lHom(V,\bC)
=\bop_{n\in\La}\Hom(V_{-n},\bC)\in\Vect^\La$.
If $V$ and $W$ are finite-dimensional, then
\begin{gather}
\chr(V[n])=\sum_{m\in\La} \dim V_{m-n}\cdot e^m=\chr(V)\cdot e^{n},\\
\chr\rbr{\lHom(V,W)}
=\chr(DV)\cdot\chr(W)
=\chr(V)^*\cdot\chr(W).
\end{gather}

Let $\tl M\in\mmod\Pi(\tl\Ga)$ and $M=\pi_*(\tl M)$.
For every vertex $i\in\Ga_0$, 
we can consider $M_i=\bop_{n\in\La}\tl M_{i,n}$ as a \La-graded vector space.
For every arrow $a:i\to j$ in \Ga and $n\in\La$,
we have the map $\tl M_{a,n}:\tl M_{i,n-\bbd_a}\to \tl M_{j,n}$, hence the map $M_a:M_i\to M_j[-\bd_a]$ has degree zero.
Recall from ~\eqref{ta-shifted rep} that we define the module $M^\ta$ with $M^\ta_i=M_{\ta i}$ and $M^\ta_a=M_{\ta a}$ for $a:i\to j$ in $\Ga$.
Note that we have degree zero maps $M^\ta_a:M^\ta_i\to M^\ta_j[-\bd_a]$ as $\bd_{\ta a}=\bd_a$.

Let us consider another module 
$\tl N\in\mmod\Pi(\tl\Ga)$ and let $N=\pi_*(\tl N)$.
Then the complex from Proposition ~\ref{prop:ex seq}
can be considered as a complex in $\Vect^\La$ (with differentials having \La-degree zero)
\begin{multline}
\label{grad-exact-seq}
\dots\to0\to\bop_{i\in\Ga_0}\lHom(M_i,N_i)
\to\bop_{(a:i\to j)\in\Ga_1}\lHom(M_i,N_j)[-\bbd_a]\\
\to\bop_{i\in\Ga_0}\lHom(M_{\ta i},N_i)[-\om]
\to0\to\dots
\end{multline}
where $\om=\bbd_a+\bbd_{\si a}$ is the total weight of $\bd:\Ga_1\to\La$, independent of $a\in\Ga_1$ (see \S\ref{localiz}).

Let $Q\sbs\Ga$ be a cut and let us define
\begin{equation}
\chr_Q^\bd(M,N)
=\sum_{i\in Q_0}\chr(M_i)^*\cdot \chr(N_i)
-\sum_{(a:i\to j)\in Q_1}
\chr(M_i)^*\cdot\chr(N_j)\cdot e^{-\bbd_a}\in\bZ[\La].
\end{equation}

The following result can be compared to Proposition \ref{h^i}.

\begin{lemma}
The character of the complex \eqref{grad-exact-seq} is equal to
$$\chr_Q^\bd(M,N)+\chr_Q^\bd(N,M^\ta)^*\cdot e^{-\om}.$$
\end{lemma}
\begin{proof}
Let $a:i\to j$ be an arrow in $Q_1$ and $\si a:\ta j\to i$ be the corresponding arrow in $\si Q_1$.
Then the character of the summand 
$\lHom(M_{\ta j}, N_{i})[-\bd_{\si a}]$ is equal to
$$\chr(M_{\ta j})^*\chr(N_i)e^{-\bd_{\si a}}
=(\chr(N_i)^*\chr(M^\ta_j)e^{-\bd_a})^*\cdot e^{-\om}.
$$
Similarly, the character of the summand
$\lHom(M_{\ta i}, N_{i})[-\om]$ is equal to
$$\chr(M_{\ta i})^*\chr(N_i)e^{-\om}
=(\chr(N_i)^*\chr(M^\ta_i))^*\cdot e^{-\om}.
$$
These observations imply the required formula.
\end{proof}

\begin{theorem}
\label{th:character}
Let $\tl M\in\cM^\st_{\tl\te}(\Pi(\tl\Ga),\tl\bv)$ and $M=\pi_*(\tl M)\in \cM^T$,
where $\cM=\cM^\st_{\te}(\Pi(\Ga),\bv)$.
Then the character of the tangent space $T_M\cM$ (considered as a $T$-representation) satisfies
$$1-\chr(T_M\cM)
=\chr_Q^\bd(M,M)+\chr_Q^\bd(M,M^\ta)^*\cdot e^{-\om}.
$$
\end{theorem}
\begin{proof}
We have $T_M\cM\iso\Ext^1_\Pi(M,M)$.
By Proposition \ref{prop:ex seq},
the cohomologies of the complex ~\eqref{grad-exact-seq} (for $N=M$)
are equal to $\Hom_\Pi(M,M)$, $\Ext^1_\Pi(M,M)$ and $D\Hom_\Pi(M,M^\ta)$ (considered as $\La$-graded vector spaces).
By our assumptions, we have $\Hom_\Pi(M,M)\iso\bC$, $\Hom_\Pi(M,M^\ta)=0$, hence 
the character of the complex is equal
to $1-\chr(T_M\cM)$ and we apply the previous lemma.
\end{proof}

\begin{remark}
Given $\tl\bv\in\bN^{(\Ga_0\xx\La)}$ and $i\in\Ga_0$,
we define 
$$\tl\bv_i=(\tl\bv_{i,n})_{n\in\La},\qquad
\chr(\tl\bv_i)=\sum_{n\in\La}\tl\bv_{i,n}e^n\in\bZ[\La].$$
Then the above result implies that
for $\tl M\in\cM^\st_{\tl\te}(\Pi(\tl\Ga),\tl\bv)$ and $M=\pi_*(\tl M)\in\cM^T$, we have
$$
1-\chr(T_M\cM)=\sum_{i\in\Ga_0}
\chr(\tl\bv_i)\rbr{
\chr(\tl\bv_i)^*+\chr(\tl\bv_{\ta i})^*e^{-\om}}
+\sum_{(a:i\to j)\in\Ga_1}
\chr(\tl\bv_i)^*\chr(\tl\bv_j)e^{-\bd_a}.$$
Therefore the character of $T_M\cM$ depends just on the dimension vector $\tl\bv$.
\end{remark}

\sec[Decomposition of translation quiver varieties]
In view of \S\ref{birula} we will study attractors of torus actions on translation quiver varieties.
The fixed locus was already determined in \S\ref{fixed locus}.
Let $\Ga$ be a translation quiver with a cut and
$\Pi=\Pi(\Ga)$ be the corresponding mesh algebra.
Let $\bv\in\ds\bN{\Ga_0}$ be a dimension vector and $\te\in\bR^{\Ga_0}$ be a stability parameter.
The subspace $\cL_\te(\Pi,\bv)\sbs \cM_\te(\Pi,\bv)$ of nilpotent representations is equal to  $\pi\inv(0)$,
where
$$\pi:\cM_\te(\Pi,\bv)\to \cM_0(\Pi,\bv)$$
is the projective morphism discussed in Remark \ref{projection to semisimp}.
Consider a weight map $\bd:\Ga_1\to\bZ$ with $\bbd_a>0$ for all $a\in\Ga_1$ and the corresponding action of $T=\Gm$ on the translation quiver variety $\cM_\te(\Pi,\bv)$.

\begin{proposition}
We have
\begin{enumerate}
\item The space $\cM_\te(\Pi,\bv)^T$ is projective.
\item For any $M\in \cM_\te(\Pi,\bv)$, there exists $\lim_{t\to 0}tM$.
\item The attractor $\sets{M\in \cM_\te(\Pi,\bv)}{\exists\,\lim_{t\to \infty}tM}$ is equal to $\cL_\te(\Pi,\bv)$.
\end{enumerate}
\end{proposition}
\begin{proof}
\clm1 We define the action of $T=\Gm$ on $\cR(\Pi,\bv)$ in the same way as on $\cM(\Pi,\bv)$.
By assumption, for any $M\in \cR(\Pi,\bv)$, we have $\lim_{t\to0}tM=0$.
Let 
$p:\cR(\Pi,\bv)\to \cM_0(\Pi,\bv)$
be the projection, which is $T$-equivariant.
If $p(M)$ is $T$-fixed, then $p(M)=\lim_{t\to0}p(tM)=0$. 
Therefore $\cM_0(\Pi,\bv)^T=\set0$.
The space $\cM_\te(\Pi,\bv)^T$ is mapped to $\cM_0(\Pi,\bv)^T=\set0$ by $\pi$, hence
$\cM_\te(\Pi,\bv)^T\sbs\pi\inv(0)$ is projective.

\clm2 For any $M\in\cM_\te(\Pi,\bv)$, we have $\lim_{t\to0}\pi(tM)=0$, hence the map $\Gm\to\cM_0(\Pi,\bv)$, $t\to\pi(tM)$ extends to $\bA^1\to\cM_0(\Pi,\bv)$.
As $\pi$ is projective, we conclude by the  valuative criterion that $\Gm\to\cM_\te(\Pi,\bv)$, $t\mto tM$, extends to
$\bA^1\to\cM_\te(\Pi,\bv)$, hence the limit $\lim_{t\to0}tM$ exists.

\clm3 If $\lim_{t\to\infty}tM$ exists, then $\lim_{t\to\infty}t \pi(M)$ exists in $\cM_0(\Pi,\bv)$.
We have seen that $\lim_{t\to0}t\pi(M)$ is zero.
Therefore we obtain a morphism $\bP^1\to\cM_0(\Pi,\bv)$ to an affine variety $\cM_0(\Pi,\bv)$.
This map has to be constantly zero, hence $\pi(M)=0$ and $M\in\pi\inv(0)=\cL_\te(\Pi,\bv)$.
Conversely, if $M\in\pi\inv(0)$, then $\Gm\to\cM_0(\Pi,\bv)$, $t\to \pi(tM)=0$, extends to $\bP^1\to\cM_0(\Pi,\bv)$.
As $\pi$ is projective, the map $\Gm\to \cM_\te(\Pi,\bv)$, $t\to tM$, extends to $\bP^1\to \cM_\te(\Pi,\bv)$.
\end{proof}

This result implies that we have attractors
$$\cM_\te(\Pi,\bv)^+=\cM_\te(\Pi,\bv),\qquad
\cM_\te(\Pi,\bv)^-=\cL_\te(\Pi,\bv).$$

\begin{corollary}
\label{dual of L}
Assume that $\te$ is $\bv$-generic, $\te^\ta=\te$ and $\bv^\ta\ne\bv$.
Then 
$\cM=\cM_\te(\Pi,\bv)$ and
$\cL=\cL_\te(\Pi,\bv)$ are pure
and their motivic classes satisfy
$$[\cL]\dual=\bL^{-\dim\cM}[\cM]$$
where
$\dim\cM=1-\hi(\bv,\bv+\bv^\ta)$ and $\hi$ is the Euler-Ringel form of the cut $\Ga^+$.
Similarly,
$$P(\cL;t\inv)=t^{-\dim\cM}P(\cM;t).$$
\end{corollary}
\begin{proof}
We proved in Theorem \ref{smoothness2} that $\cM_\te(\Pi,\bv)=\cM_\te^\s(\Pi,\bv)$ is smooth and has the stated dimension.
As $\cM_\te(\Pi,\bv)^T$ is projective by the previous result, we can apply Proposition \ref{dual1} and Proposition~ \ref{pure}.
\end{proof}

\begin{theorem}
\label{tate class}
Assume that $\te$ is $\bv$-generic, $\te^\ta=\te$ and $\bv^\ta\ne\bv$.
Then $\cM_\te(\Pi,\bv)$ and $\cL_\te(\Pi,\bv)$ have Tate
motivic classes.
\end{theorem}
\begin{proof}
As before, we consider the action of $T=\Gm$ on $\cM=\cM_\te(\Pi,\bv)$.
By the previous results we obtain the \BB decompositions of $\cM^+=\cM_\te(\Pi,\bv)$ and $\cM^-=\cL_\te(\Pi,\bv)$.
All strata of these decompositions are affine bundles over the connected components of $\cM^T$.
By Theorem \ref{decomp} we can decompose $\cM^T$ as a disjoint union of the translation quiver varieties
$\cM^\s_{\tl\te}(\Pi(\tl\Ga),\tl\bv)$ for the  localization quiver $\tl\Ga=\loc_\bd(\Ga)$.
The dimensions of the above affine bundles over 
$\cM^\s_{\tl\te}(\Pi(\tl\Ga),\tl\bv)$ depend just on $\tl\bv$.
Indeed, for any $\tl M\in\cM^\s_{\tl\te}(\Pi(\tl\Ga),\tl\bv)$ and $M=\bar\io(\tl M)\in\cM^T$, the dimension of the $+$-attractor (see \S\ref{birula}) over $M$ is equal to the dimension of the positive component of the tangent space $T_M\cM$.
We proved in Theorem \ref{th:character} that the $T$-character of $T_M\cM$ depends just on ~$\tl\bv$.
In particular, the dimension of the positive component of $T_M\cM$ depends just on ~$\tl\bv$.
The same argument applies to the $-$-attractor.
We conclude that to prove the theorem we need to show that
the translation quiver varieties
$\cM^\s_{\tl\te}(\Pi(\tl\Ga),\tl\bv)$
have Tate motivic classes.
We will prove that they have Tate motivic classes for a special choice of the weight map $\bd$.

Assume that $\Ga=Q^\ta$, for a quiver with an automorphism $(Q,\ta)$.
Let us choose $r>\n\bv=\sum_i \bbv_i$ and choose $\bd:Q_1\to\bZ$ with $\bd_a=r$ for all $a$.
Let $\bd:\Ga_1\to\bZ$ also denote the corresponding weight map having the total weight $\om=r+1$.
Then $\tl\Ga=\loc_{\bd}(\Ga)\iso (\tl Q)^{\ta_\om}$ (see \S\ref{localiz}), where
$\tl Q=\loc_\bd(Q)$ and
$$\ta_\om:\tl Q\to\tl Q,\qquad (i,n)\mto(\ta i,n-\om),
\qquad a_n\mto (\ta a)_{n-\om}.$$
Recall that $\tl Q$ has arrows $a_n:(i,n-r)\to(j,n)$ for all $a:i\to j$ in $Q$ and $n\in\bZ$.
For every $k\in\bZ_r$, consider 
$$Q^{(k)}_0=\sets{(i,n)\in\tl Q_0}{n\equiv k\pmod r}$$
and let $Q^{(k)}\sbs \tl Q$ be the corresponding full subquiver.
There are no arrows between different quivers $Q^{(k)}$ in $\tl Q$.
The translation $\ta_\om$
induces isomorphisms $\ta_\om:Q^{(k)}\to Q^{(k-1)}$.
In the quiver $(\tl Q)^{\ta_\om}$ we have additional arrows
$$a_n^*:(\ta j,n-\om)\to (i,n-r)$$
for $a:i\to j$ and $n\in\bZ$.
This means that $a_n^*$ is an arrow from $Q^{(k-1)}$ to $Q^{(k)}$ for $n\equiv k\pmod r$.

As \te is \bv-generic, a torus fixed point in $\cM_\te(\Pi(\Ga),\bv)=\cM_\te^\st(\Pi(\Ga),\bv)$ corresponds to a point $M\in\cM_{\tl\te}^\st(\Pi(\tl\Ga),\tl\bv)$, where 
$\tl\bv$ is a dimension vector on $\tl\Ga=\loc_\bd(\Ga)$ with $\pi_*(\tl\bv)=\bv$,
$\tl\te=\pi^*(\te)$ and $\pi:\tl\Ga\to\Ga$ is the projection.
For every $0<\tl\bu<\tl\bv$, we have $$\mu_{\tl\te}(\tl\bu)=\mu_\te(\bu)\ne\mu_\te(\bv)=\mu_{\tl\te}(\tl\bv),\qquad
\bu=\pi_*(\tl\bu)<\bv.$$
This implies that $\tl\te$ is $\tl\bv$-generic,
hence $\cM_{\tl\te}(\Pi(\tl\Ga),\tl\bv)=\cM_{\tl\te}^\st(\Pi(\tl\Ga),\tl\bv)$.

As $r>\n\bv=\dim M$, the representation $M$ is not supported on some $Q^{(k)}$.
Without loss of generality we can assume that $M$ is not supported on the quiver $R=Q^{(0)}$.
We can substitute translations (which are quiver isomorphisms)
$$Q^{(r-1)}\xto{\ta_\om}\dots\xto{\ta_\om}Q^{(0)}=R$$
by identities (the last translation $Q^{(0)}\to Q^{(r-1)}$ will be possibly not an identity, but it is irrelevant for us as $M$ is not supported on $Q^{(0)}$).
Therefore we can consider $M$ as a (semi-)stable representation of
the repetition quiver $\hat\Ga=\bZ R$, where
we lift the stability parameter from $R=Q^{(0)}$ to $\bZ R$.
Recall that $\bZ R\iso(R\xx\bZ)^{\ta}$, where 
$$\ta(i,k)=(i,k-1),\qquad \ta(a,k)=(a,k-1),\qquad
i\in R_0,\ a\in R_1,\ k\in\bZ.$$
We will write $\hat\te$ and $\hat\bv$ for the stability parameter and the dimension vector on $\hat\Ga=\bZ R$.
An object in $M\in \cM_{\hat\te}(\Pi(\hat\Ga),\hat\bv)$ is stable, hence its automorphism group is $\Gm$.
This implies that we have a $\Gm$-gerbe
$\fM_{\hat\te}(\Pi(\hat\Ga),\hat\bv)
\to\cM_{\hat\te}(\Pi(\hat\Ga),\hat\bv)
$, hence
$$[\cM_{\hat\te}(\Pi(\hat\Ga),\hat\bv)]=(\bL-1)\cdot
[\fM_{\hat\te}(\Pi(\hat\Ga),\hat\bv)].$$
Note that $\hat\Ga=\bZ R=\loc_0^1(\bar R)$
(see Remark \ref{rep as localiz}),
hence by Corollary \ref{quasi-tate stack} we conclude that
$\cM_{\hat\te}(\Pi(\hat\Ga),\hat\bv)$ has a quasi-Tate motivic class.
But $\cM_{\hat\te}(\Pi(\hat\Ga),\hat\bv)$ is an algebraic variety, hence it has a Tate motivic class (see Remark \ref{int coeff}).
\end{proof}


Finally, let us consider (framed) translation quiver varieties.
Given $\bw\in\bN^{\Ga_0}$, we construct the framed quiver $\Ga^\f$ (which is a partial translation quiver) and the stability parameter $\tef$ on it (see ~\S\ref{translation qv}).
Given a dimension vector $\bv\in\ds\bN{\Ga_0}$,
we extend it to $\bvf\in\bN^{(\Ga^\f_0)}$ by setting $\bvf_*=1$.
Then we define
$$\cM(\bv,\bw)=\cM_{\tef}(\Pi(\Ga^\f),\bvf),\qquad
\cL(\bv,\bw)=\cL_{\tef}(\Pi(\Ga^\f),\bvf).$$

\begin{corollary}
\label{tate framed qv}
Translation quiver varieties $\cM(\bv,\bw)$ and $\cL(\bv,\bw)$ are pure and have Tate motivic classes.
\end{corollary}
\begin{proof}
We have seen in Theorem \ref{framed smooth}
that $\cM(\bv,\bw)$ can be interpreted as a translation quiver variety for the stable translation quiver $\hGaf$.
Therefore the result follows from the previous theorem.
\end{proof}

%% file: nak_induction.tex
\subsection{Alternative approach}
There is an alternative way, based on the algorithm from \cite{nakajima_quivere},
to show that translation quiver varieties $\cM(\bv,\bw)$
have Tate motivic classes. 
Let~\Ga be a translation quiver with a cut, $\Pi=\Pi(\Ga)$ be its mesh algebra and $*\in\Ga_0$ be a distinguished vertex.
Let $\bv$ be a dimension vector with $\bv_*=1$ and let $\te$ be a stability parameter with $\te_*=1$ and $\te_i=0$ for all $i\ne *$.
Note that $\te$ is $\bv$-generic and a representation $M$ having dimension vector $\bv$ is $\te$-stable if and only if it is generated by~$M_*$.
We will denote the moduli space $\cM_\te(\Pi,\bv)$ by $\cM_*(\bv)$.

\begin{proposition}
Let $\bv\in\bN^{(\Ga_0)}$ be a dimension vector such that $\bv_*=1$ and $\bv_{\ta *}=0$.
Then $\cM_*(\bv)$ is smooth.
\end{proposition}
\begin{proof}
We apply the same argument as in Theorem \ref{smoothness2}.
For any $M\in\cM_*(\bv)$, we have
$$h^0(M,M)-h^1(M,M)+h^0(M,M^\ta)=\hi(M,M)+\hi(M,M^\ta),$$
where $\hi$ is the Euler-Ringel form of the cut.
We note that $h^0(M,M)=1$ as $M$ is stable.
On the other hand $h^0(M,M^\ta)=0$ as $M$ is generated by $M_*$ and $M^\ta_*=M_{\ta *}=0$.
This implies $h^1(M,M)=1-\hi(\bv,\bv+\bv^\ta)$.
\end{proof}

From now on we will assume that $\bv_*=1$ and $\bv_{\ta^k*}=0$ for $k\ne0$.
Performing localization, we can assume that the translation quiver $\Ga$ is acyclic.
Then there are no arrows $i\to \ta^ki$ for $k\in\bZ$.
Indeed, otherwise there exists an arrow $\ta^{k+1}i\to i$, hence an arrow $i\to\ta^{-k-1}i$ and we can assume that $k>0$.
Then there exists a path $\ta i\to\ta^{k+1}i\to i$, hence also a path $\ta^ki\to i$ and a cycle $\ta^ki\to i\to\ta^k i$
which contradicts to the assumption that $\Ga$ is acyclic.

For any $M\in\cM_*(\bv)$ and a vertex  $i\in\Ga_0$, consider the complex
\begin{gather*}
M_{\ta i}\xto {f_{\ta i}}\bop_{a\rcol j\to i}M_j\xto {g_i} M_i,\\
f_{\ta i}=\bop_{a\rcol j\to i}M_{\si a},\qquad g_i=\sum_{a\rcol j\to i}\eps(a) M_a.
\end{gather*}
Note that $g_if_{\ta i}=0$ because of the mesh relation and $g_i$ is surjective (for $i\ne *$) because of the stability condition.
Given a \ta-orbit $\cO\sbs\Ga_0$ with $*\notin\cO$ and a vector  $\br\in\bN^\cO$, 
we define
$$\cM_{*,\br}(\bv)=\sets{M\in \cM_*(\bv)}{\dim\Ker f_{i}=\br_i\ \forall i\in \cO}.$$
For any $M\in \cM_{*,\br}(\bv)$, there is a subrepresentation $N\sbs M$ defined by $N_i=\Ker f_i$ for $i\in\cO$ and $N_i=0$ otherwise.
The quotient $M'=M/N$ is stable and is contained in $\cM_*(\bv-\br)$, where 
$$\bv-\br=\udim M'=\bv-\sum_{i\in \cO} \br_ie_i.$$
It is actually contained in $\cM_{*,0_{\cO}}(\bv-\br)$, where $0_\cO\in\bN^{\cO}$ is the zero vector, as we have injective maps
$$M'_{\ta i}\xto{f'_{\ta i}}\bop_{a\rcol j\to i}M_j.$$
Possible stable representations $M$ that extend $M'$
correspond to surjective maps
$$\bar g_{i}:\Coker f'_{\ta i}
\to\Ker(f_i),\qquad i\in\cO.
$$
We have
$$d_i(\bv,\br):=\dim\Coker f'_{\ta i}=\sum_{a\rcol j\to i}\bv_j-\bv_{\ta i}+\br_{\ta i},$$
hence
\begin{equation}
\label{ind motive}
[\cM_*(\bv)]
=\sum_{\br\le\bv|_{\cO}}
[\cM_{*,\br}(\bv)]
=\sum_{\br\le\bv|_{\cO}}
[\cM_{*,0_\cO}(\bv-\br)]\cdot \prod_{i\in\cO}[\Gr(d_i(\bv,\br),\br_i)],
\end{equation}
where $\Gr(n,r)$ is the Grassmannian parametrizing $r$-dimensional quotients of an $n$-dimensional space.
Using the above formula we can recursively express $[\cM_{*,0_\cO}(\bv)]$ as a linear combination (with coefficients that are Tate motivic classes) of $[\cM_*(\bu)]$ with $\bu\le\bv$.

For $\bv=e_*$, the moduli space $\cM_*(\bv)$ is just a point.
Let us assume that $\bv>e_*$ and that $\cM_*(\bu)$ have Tate motivic classes for all $\bu<\bv$.
Let $i\in\Ga_0$ be such that $\bv_i\ne0$ and $\bv_j=0$ for all arrows $a:i\to j$ (such vertex exists as $\Ga$ is acyclic by our assumption).
We can assume that $i\ne*$ as otherwise $\cM_*(\bv)=\es$.
Let ~$\cO\sbs\Ga_0$ be the \ta-orbit of the vertex $i$.
By the previous discussion,
the motivic classes of $\cM_{*,0_\cO}(\bu)$ are Tate,  for all $\bu<\bv$.
We have $\cM_{*,0_\cO}(\bv)=\es$ as $f_i=0$ by our assumption.
We conclude from \eqref{ind motive} that $\cM_*(\bv)$ has a Tate motivic class.

%% file: mot_classes_QV.tex
\subsection{Motivic classes of Nakajima quiver varieties}
Let us consider a finite quiver ~$Q$, the double quiver $\Ga=\bar Q$ and vectors $\bv,\bw\in\bN^{Q_0}$.
In this case translation quiver varieties $\cM(\bv,\bw)$ and $\cL(\bv,\bw)$ are exactly the Nakajima quiver varieties.
The virtual Poincar\'e polynomials of $\cM(\bv,\bw)$ were computed in \cite{hausel_kac,mozgovoy_fermionic} by counting points of $\cM(\bv,\bw)$ over finite fields.
It was proved in \cite{wyss_motivic} that the same formula is satisfied by the motivic classes of $\cM(\bv,\bw)$.
The virtual Poincar\'e polynomials of $\cL(\bv,\bw)$ were computed in \cite{bozec_number} by applying \BB decomposition to a torus action on $\cM(\bv,\bw)$ (\cf Corollary \ref{virtual1}).
Because of Theorem \ref{tate framed qv} we can write motivic classes of $\cM(\bv,\bw)$ and $\cL(\bv,\bw)$ automatically, if we know their Poincar\'e polynomials (or if we can count their points over finite fields).
Let us write down these formulas for completeness (\cf \cite{mozgovoy_fermionic,bozec_number}).
We define
$$r(\bw,q\inv,z)=\prod_{\ta}q^{-\bw\cdot\ta_1}\prod_{k\ge1}q^{\hi(\ta_k,\ta_k)}\frac{z^{\ta_k}}{(q)_{\ta_k-\ta_{k+1}}}$$
where
\begin{enumerate}
\item  $\ta=(\ta^i)_{i\in Q_0}$ is a collection of partitions,
\item $\ta_k=(\ta^i_k)_{i\in Q_0}\in\bN^{Q_0}$ for $k\ge1$, \item $z^\bv=\prod_{i\in Q_0}z_i^{v_i}$ for $\bv\in\bN^{Q_0}$,
\item $(q)_{\bv}=\prod_{i\in Q_0}(q)_{v_i}$, $(q)_n=(q;q)_n=\prod_{k=1}^n(1-q^k)$ for $\bv\in\bN^{Q_0}$ and $n\in\bN$,
\item $\hi$ is the Euler-Ringel form of the quiver $Q$.
\end{enumerate}

\begin{theorem}
We have
$$\sum_{\bv\in\bN^{Q_0}}\bL^{-d(\bv,\bw)}[\cM(\bv,\bw)]z^\bv=\frac{r(\bw,\bL,z)}{r(0,\bL,z)},$$
$$\sum_{\bv\in\bN^{Q_0}}\bL^{-d(\bv,\bw)}[\cL(\bv,\bw)]z^\bv=\frac{r(\bw,\bL\inv,z)}{r(0,\bL\inv,z)},$$
where $d(\bv,\bw)=\oh\dim\cM(\bv,\bw)=\bv\cdot\bw-\hi(\bv,\bv)$.
\end{theorem}
\begin{proof}
For the first formula see \cite{hausel_kac,mozgovoy_fermionic,wyss_motivic}.
For the second formula note that by Corollary~\ref{dual of L} we have
$$[\cL(\bv,\bw)]\dual=\bL^{-2d(\bv,\bw)}[\cM(\bv,\bw)].$$
Therefore
$$\sum_{\bv\in\bN^{Q_0}}\bL^{d(\bv,\bw)}[\cL(\bv,\bw)]\dual z^\bv=\frac{r(\bw,\bL,z)}{r(0,\bL,z)}.$$
Taking the duals we obtain the required result (\cf \cite{bozec_number}).
\end{proof}